\documentclass[journal]{IEEEtran}
\usepackage{cite}
\usepackage{graphicx}
\usepackage{amsfonts}
\usepackage{amssymb}
\usepackage{amsmath}
\usepackage{bm}
\usepackage{mathtools}
\usepackage{amsthm}
\usepackage{algorithm}
\usepackage{algpseudocode}
\usepackage{array}
\usepackage{stfloats}
\usepackage{url}
\usepackage{xcolor}
\usepackage{bbm}
\usepackage{multirow}
\usepackage{booktabs}
\usepackage{psfrag}
\usepackage{subcaption}
\usepackage{tikz}
\usepackage[linkcolor=black, citecolor=black, urlcolor = blue, hidelinks]{hyperref}

\allowdisplaybreaks

\newtheorem{theorem}{Theorem}
\newtheorem*{remark}{Remark}

\DeclareMathOperator{\tr}{Tr}
\DeclareMathOperator{\E}{\mathbb{E}}

\newcommand\copyrighttext{%
  \footnotesize \textcopyright 2018 IEEE. Personal use of this material is permitted.
  Permission from IEEE must be obtained for all other uses, in any current or future
  media, including reprinting/republishing this material for advertising or promotional
  purposes, creating new collective works, for resale or redistribution to servers or
  lists, or reuse of any copyrighted component of this work in other works.
  DOI: \href{https://doi.org/10.1109/TSP.2018.2866385}{10.1109/TSP.2018.2866385}, IEEE
Transactions on Signal Processing}
\newcommand\copyrightnotice{%
\begin{tikzpicture}[remember picture,overlay]
\node[anchor=south, yshift=3pt] at (current page.south) {\fbox{\parbox{\dimexpr\textwidth-\fboxsep-\fboxrule\relax}{\copyrighttext}}};
\end{tikzpicture}%
}

\begin{document}

\title{Bayesian Cluster Enumeration Criterion for Unsupervised Learning}

\author{Freweyni~K.~Teklehaymanot,~\IEEEmembership{Student~Member,~IEEE,}
        Michael~Muma,~\IEEEmembership{Member,~IEEE,}
        and~Abdelhak~M.~Zoubir,~\IEEEmembership{Fellow,~IEEE}% <-this % stops a space
\thanks{F. K. Teklehaymanot and A. M. Zoubir are with the Signal Processing Group and the Graduate School of Computational Engineering, Technische Universit\"at Darmstadt, Darmstadt, Germany (e-mail: ftekle@spg.tu-darmstadt.de; zoubir@spg.tu-darmstadt.de).}% <-this % stops a space
\thanks{M. Muma is with the Signal Processing Group, Technische Universit\"at Darmstadt, Darmstadt, Germany (e-mail: muma@spg.tu-darmstadt.de).}
}

\maketitle

\begin{abstract}
We derive a new Bayesian Information Criterion (BIC) by formulating the problem of estimating the number of clusters in an observed data set as maximization of the posterior probability of the candidate models. Given that some mild assumptions are satisfied, we provide a general BIC expression for a broad class of data distributions. This serves as a starting point when deriving the BIC for specific distributions. Along this line, we provide a closed-form BIC expression for multivariate Gaussian distributed variables. We show that incorporating the data structure of the clustering problem into the derivation of the BIC results in an expression whose penalty term is different from that of the original BIC. We propose a two-step cluster enumeration algorithm. First, a model-based unsupervised learning algorithm partitions the data according to a given set of candidate models. Subsequently, the number of clusters is determined as the one associated with the model for which the proposed BIC is maximal. The performance of the proposed two-step algorithm is tested using synthetic and real data sets.
\end{abstract}

% arXiv copyright notice
\copyrightnotice


%%%%%%%%%%%%%%%%%%%%%%%%%%%%
\section{Introduction}
\label{sec:intro}
\IEEEPARstart{S}{tatistical} model selection is concerned with choosing a model that adequately explains the observations from a family of candidate models. Many methods have been proposed in the literature, see for example \cite{jeffreys1961,akaike1969,akaike1970,akaike1973,allen1974,stone1974,rissanen1978,schwarz1978,hannan1979,shabata1980,rao1989,breiman1992,kass21995,shao1996,djuric1998,cavanaugh1999,zoubir21999,zoubir2000,brcich2002,morelande2002,spiegelhalter2002,claeskens2003,lu2013,lu22013,lu2015} and the review in \cite{rao2001}. Model selection problems arise in various applications, such as the estimation of the number of signal components \cite{djuric1998,lu2013,lu22013,lu2015,zoubir2000,morelande2002,brcich2002}, the selection of the number of non-zero regression parameters in regression analysis \cite{akaike1973,claeskens2003,shao1996,allen1974,stone1974,rao1989,breiman1992,spiegelhalter2002}, and the estimation of the number of data clusters in unsupervised learning problems \cite{kalogeratos2012,hamerly2003,pelleg2000,shahbaba2012,ishioka2005,zhao2008,zhao22008,feng2006,constantinopoulos2006,huang2017,fraley1998,mehrjou2016,dasgupta1998,campbell1997,mukherjee1998,krzanowski1988,tibshirani2001,teklehaymanot2016,binder2017}. In this paper, our focus lies on the derivation of a Bayesian model selection criterion for cluster analysis. \\
The estimation of the number of clusters, also called cluster enumeration, has been intensively researched for decades \cite{kalogeratos2012,hamerly2003,pelleg2000,shahbaba2012,ishioka2005,zhao2008,zhao22008,feng2006,constantinopoulos2006,huang2017,fraley1998,mehrjou2016,dasgupta1998,campbell1997,mukherjee1998,krzanowski1988,tibshirani2001,teklehaymanot2016,binder2017} and a popular approach is to apply the Bayesian Information Criterion (BIC) \cite{ishioka2005,fraley1998,zhao2008,zhao22008,pelleg2000,mehrjou2016,dasgupta1998,campbell1997,mukherjee1998,teklehaymanot2016}. The BIC finds the large sample limit of the Bayes' estimator which leads to the selection of a model that is a posteriori most probable. It is consistent if the true data generating model belongs to the family of candidate models under investigation. The BIC was originally derived by Schwarz in \cite{schwarz1978} assuming that (i) the observations are independent and identically distributed (iid), (ii) they arise from an exponential family of distributions, and (iii) the candidate models are linear in parameters. Ignoring these rather restrictive assumptions, the BIC has been used in a much larger scope of model selection problems. A justification of the widespread applicability of the BIC was provided in \cite{cavanaugh1999} by generalizing Schwarz's derivation. In \cite{cavanaugh1999}, the authors drop the first two assumptions made by Schwarz given that some regularity conditions are satisfied.  The BIC is a generic criterion in the sense that it does not incorporate information regarding the specific model selection problem at hand. As a result, it penalizes two structurally different models the same way if they have the same number of unknown parameters. \\
The works in \cite{stoica2004,djuric1998} have shown that model selection rules that penalize for model complexity have to be examined carefully before they are applied to specific model selection problems. Nevertheless, despite the widespread use of the BIC for cluster enumeration \cite{ishioka2005,fraley1998,zhao2008,zhao22008,pelleg2000,mehrjou2016,dasgupta1998,campbell1997,mukherjee1998,teklehaymanot2016}, very little effort has been made to check the appropriateness of the original BIC formulation \cite{cavanaugh1999} for cluster analysis. One noticeable work towards this direction was made in \cite{mehrjou2016} by providing a more accurate approximation to the marginal likelihood for small sample sizes. This derivation was made specifically for mixture models assuming that they are well separated. The resulting expression contains the original BIC term plus some additional terms that are based on the mixing probability and the Fisher Information Matrix (FIM) of each partition. The method proposed in \cite{mehrjou2016} requires the calculation of the FIM for each cluster in each candidate model, which is computationally very expensive and impractical in real world applications with high dimensional data. This greatly limits the applicability of the cluster enumeration method proposed in \cite{mehrjou2016}. Other than the above mentioned work, to the best of our knowledge, no one has thoroughly investigated the derivation of the BIC for cluster analysis using large sample approximations. \\
We derive a new BIC by formulating the problem of estimating the number of partitions (clusters) in an observed data set as maximization of the posterior probability of the candidate models. Under some mild assumptions, we provide a general expression for the BIC, $\text{BIC}_{\mbox{\tiny G}}(\cdot)$, which is applicable to a broad class of data distributions. This serves as a starting point when deriving the BIC for specific data distributions in cluster analysis. Along this line, we simplify $\text{BIC}_{\mbox{\tiny G}}(\cdot)$ by imposing an assumption on the data distribution. A closed-form expression, $\text{BIC}_{\mbox{\tiny N}}(\cdot)$, is derived assuming that the data set is distributed as a multivariate Gaussian. The derived model selection criterion, $\text{BIC}_{\mbox{\tiny N}}(\cdot)$, is based on large sample approximations and it does not require the calculation of the FIM. This renders our criterion computationally cheap and practical compared to the criterion presented in \cite{mehrjou2016}. \\
Standard clustering methods, such as the Expectation-Maximization (EM) and K-means algorithm, can be used to cluster data only when the number of clusters is supplied by the user. To mitigate this shortcoming, we propose a two-step cluster enumeration algorithm which provides a principled way of estimating the number of clusters by utilizing existing clustering algorithms. The proposed two-step algorithm uses a model-based unsupervised learning algorithm to partition the observed data into the number of clusters provided by the candidate model prior to the calculation of $\text{BIC}_{\mbox{\tiny N}}(\cdot)$ for that particular model. We use the EM algorithm, which is a model-based unsupervised learning algorithm, because it is suitable for Gaussian mixture models and this complies with the Gaussianity assumption made by $\text{BIC}_{\mbox{\tiny N}}(\cdot)$. However, the model selection criterion that we propose is more general and can be used as a wrapper around any clustering algorithm, see \cite{xu2005} for a survey of clustering methods. \\
The paper is organized as follows. Section~\ref{sec:probform} formulates the problem of estimating the number of clusters given data. The proposed generic Bayesian cluster enumeration criterion, $\text{BIC}_{\mbox{\tiny G}}(\cdot)$, is introduced in Section~\ref{sec:propcrit}. Section~\ref{sec:gauss} presents the proposed Bayesian cluster enumeration algorithm for multivariate Gaussian data in detail. A brief description of the existing BIC-based cluster enumeration methods is given in Section~\ref{sec:stateoftheart}. Section~\ref{sec:penTerm} provides a comparison of the penalty terms of different cluster enumeration criteria. A detailed performance evaluation of the proposed criterion and comparisons to existing BIC-based cluster enumeration criteria using simulated and real data sets are given in Section~\ref{sec:results}. Finally, concluding remarks are drawn and future directions are briefly discussed in Section~\ref{sec:conc}. A detailed proof is provided in Appendix \ref{app:A}, whereas Appendix \ref{app:B} contains the vector and matrix differentiation rules that we used in the derivations.  \\ 
\textbf{Notation:} Lower- and upper-case boldface letters stand for column vectors and matrices, respectively; Calligraphic letters denote sets with the exception of $\mathcal{L}$ which is used for the likelihood function; $\mathbb{R}$ represents the set of real numbers; $\mathbb{Z}^{+}$ denotes the set of positive integers; Probability density and mass functions are denoted by $f(\cdot)$ and $p(\cdot)$, respectively; $\bm{x}\sim\mathcal{N}(\bm{\mu},\bm{\Sigma})$ represents a Gaussian distributed random variable $\bm{x}$ with mean $\bm{\mu}$ and covariance matrix $\bm{\Sigma}$; $\hat{\bm{\theta}}$ stands for the estimator (or estimate) of the parameter $\bm{\theta}$; $\log$ denotes the natural logarithm; iid stands for independent and identically distributed; \textbf{(A.)} denotes an assumption, for example \textbf{(A.1)} stands for the first assumption; $\mathcal{O}(1)$ represents Landau's term which tends to a constant as the data size goes to infinity; $\bm{I}_r$ stands for an $r\times r$ identity matrix; $\bm{0}_{r\times r}$ denotes an $r\times r$ all zero matrix; $\#\mathcal{X}$ represents the cardinality of the set $\mathcal{X}$; $^\top$ stands for vector or matrix transpose; $|\bm{Y}|$ denotes the determinant of the matrix $\bm{Y}$; $\tr (\cdot)$ represents the trace of a matrix; $\otimes$ denotes the Kronecker product; $\text{vec}(\bm{Y})$ refers to the staking of the columns of an arbitrary matrix $\bm{Y}$ into one long column vector.  

%%%%%%%%%%%%%%%%%%%%%%%%%%%%%%%%%
\section{Problem Formulation}
\label{sec:probform}
Given a set of $r$-dimensional vectors $\mathcal{X}\triangleq \{\bm{x}_1,\ldots,\bm{x}_N\}$, let $\{\mathcal{X}_1,\ldots,\mathcal{X}_K\}$ be a partition of $\mathcal{X}$ into $K$ clusters $\mathcal{X}_k\subseteq\mathcal{X}$ for $k\in\mathcal{K}\triangleq\{1,\ldots,K\}$. The subsets (clusters) $\mathcal{X}_k, k\in\mathcal{K},$ are independent, mutually exclusive, and non-empty. Let $\mathcal{M}\triangleq \{M_{L_\mathrm{min}},\ldots,M_{L_\mathrm{max}}\}$ be a family of candidate models that represent a partitioning of $\mathcal{X}$ into $l=L_\mathrm{min},\ldots,L_\mathrm{max}$ subsets, where $l\in\mathbb{Z}^{+}$.
The parameters of each model $M_l\in\mathcal{M}$ are denoted by $\bm{\Theta}_l=\left[\bm{\theta}_1,\ldots,\bm{\theta}_l\right]$ which lies in a parameter space $\Omega_l\subset\mathbb{R}^{q\times l}$. Let $f(\mathcal{X}|M_l,\bm{\Theta}_l)$ denote the probability density function (pdf) of the observation set $\mathcal{X}$ given the candidate model $M_l$ and its associated parameter matrix $\bm{\Theta}_l$. Let $p(M_l)$ be the discrete prior of the model $M_l$ over the set of candidate models $\mathcal{M}$ and let $f(\bm{\Theta}_l|M_l)$ denote a prior on the parameter vectors in $\bm{\Theta}_l$ given $M_l\in\mathcal{M}$. \\
According to Bayes' theorem, the joint posterior density of $M_l$ and $\bm{\Theta}_l$ given the observed data set $\mathcal{X}$ is given by
\begin{equation}
 f(M_l,\bm{\Theta}_l|\mathcal{X}) = \frac{p(M_l)f(\bm{\Theta}_l|M_l)f(\mathcal{X}|M_l,\bm{\Theta}_l)}{f(\mathcal{X})},
\end{equation}
where $f(\mathcal{X})$ is the pdf of $\mathcal{X}$. Our objective is to choose the candidate model $M_{\hat{K}}\in\mathcal{M}$, where $\hat{K}\in\{L_\mathrm{min},\ldots,L_\mathrm{max}\}$, which is most probable a posteriori assuming that 
\begin{enumerate}
 \item[\textbf{(A.1)}] the true number of clusters $(K)$ in the observed data set $\mathcal{X}$ satisfies the constraint $L_\mathrm{min}\leq K\leq L_\mathrm{max}$.
\end{enumerate}
Mathematically, this corresponds to solving
\begin{equation}
 M_{\hat{K}} = \arg\underset{\mathcal{M}}{\max} \, \, p(M_l|\mathcal{X}),
 \label{eq:Mhat}
\end{equation}
where $p(M_l|\mathcal{X})$ is the posterior probability of $M_l$ given the observations $\mathcal{X}$. $p(M_l|\mathcal{X})$ can be written as
\begin{align}
 p(M_l|\mathcal{X}) &= \int_{\Omega_l} f(M_l,\bm{\Theta}_l|\mathcal{X})d\bm{\Theta}_l \nonumber \\
 &= f(\mathcal{X})^{-1}p(M_l)\int_{\Omega_l} f(\bm{\Theta}_l|M_l)\mathcal{L}(\bm{\Theta}_l|\mathcal{X})d\bm{\Theta}_l,
\label{eq:posterior}
\end{align} 
where $\mathcal{L}(\bm{\Theta}_l|\mathcal{X})\triangleq f(\mathcal{X}|M_l,\bm{\Theta}_l)$ is the likelihood function. $M_{\hat{K}}$ can also be determined via
\begin{equation}
 \arg \underset{\mathcal{M}}{\max} \, \, \log p(M_l|\mathcal{X})
\end{equation}
instead of Eq.~\eqref{eq:Mhat} since $\log$ is a monotonic function. Hence, taking the logarithm of Eq.~\eqref{eq:posterior} results in
\begin{equation}
 \log p(M_l|\mathcal{X}) \!=\! \log p(M_l) + \log\! \int_{\Omega_l}\!\! f(\bm{\Theta}_l|M_l)\mathcal{L}(\bm{\Theta}_l|\mathcal{X})d\bm{\Theta}_l + \rho,
 \label{eq:logposterior*}
\end{equation}
where $-\log f(\mathcal{X})$ is replaced by $\rho$ (a constant) since it is not a function of $M_l\in\mathcal{M}$ and thus has no effect on the maximization of $\log p(M_l|\mathcal{X})$ over $\mathcal{M}$. Since the partitions (clusters) $\mathcal{X}_m\subseteq\mathcal{X}, m=1,\ldots,l,$ are independent, mutually exclusive, and non-empty, $f(\bm{\Theta}_l|M_l)$ and $\mathcal{L}(\bm{\Theta}_l|\mathcal{X})$ can be written as 
\begin{align}
 f(\bm{\Theta}_l|M_l) &= \prod_{m=1}^lf(\bm{\theta}_m|M_l) \label{eq:priorpar} \\
 \mathcal{L}(\bm{\Theta}_l|\mathcal{X}) &= \prod_{m=1}^l\mathcal{L}(\bm{\theta}_m|\mathcal{X}_m).\label{eq:totalloglikeli}
\end{align}
Substituting Eqs.~\eqref{eq:priorpar} and \eqref{eq:totalloglikeli} into Eq.~\eqref{eq:logposterior*} results in 
\begin{align}
 \log p(M_l|\mathcal{X}) \!\!&=\! \log p(M_l) \!+\!\!\! \sum_{m=1}^l\!\log\! \int_{\mathbb{R}^q}\!\!\!\! f(\bm{\theta}_m|M_l)\mathcal{L}(\bm{\theta}_m|\mathcal{X}_m)d\bm{\theta}_m \! \nonumber \\
 & +\! \rho.
 \label{eq:logposterior}
\end{align}
Maximizing $\log p(M_l|\mathcal{X})$ over all candidate models $M_l\in\mathcal{M}$ involves the computation of the logarithm of a multidimensional integral. Unfortunately, the solution of the multidimensional integral does not possess a closed analytical form for most practical cases. This problem can be solved using either numerical integration or approximations that allow a closed-form solution. In the context of model selection, closed-form approximations are known to provide more insight into the problem than numerical integration \cite{djuric1998}. Following this line of argument, we use Laplace's method of integration \cite{djuric1998,stoica2004,ando2010} and provide an asymptotic approximation to the multidimensional integral in Eq.~\eqref{eq:logposterior}.  

%%%%%%%%%%%%%%%%%%%%%%%%%%%%%%%%%%%%%%%%%%%%%%%%%%%%%%%%%%%%%
\section{Proposed Bayesian Cluster Enumeration Criterion}
\label{sec:propcrit}
In this section, we derive a general BIC expression for cluster analysis, which we call $\text{BIC}_{\mbox{\tiny G}}(\cdot)$. Under some mild assumptions, we provide a closed-form expression that is applicable to a broad class of data distributions. \\
In order to provide a closed-form analytic approximation to Eq.~\eqref{eq:logposterior}, we begin by approximating the multidimensional integral using Laplace's method of integration. Laplace's method of integration makes the following assumptions.
\begin{enumerate}
 \item[\textbf{(A.2)}] $\log \mathcal{L}(\bm{\theta}_m|\mathcal{X}_m)$ with $m=1,\ldots,l$ has first- and second-order derivatives which are continuous over the parameter space $\Omega_l$.
 \item[\textbf{(A.3)}] $\log \mathcal{L}(\bm{\theta}_m|\mathcal{X}_m)$ with $m=1,\ldots,l$ has a global maximum at $\hat{\bm{\theta}}_m$, where $\hat{\bm{\theta}}_m$ is an interior point of $\Omega_l$.
 \item[\textbf{(A.4)}] $f(\bm{\theta}_m|M_l)$ with $m=1,\ldots,l$ is continuously differentiable and its first-order derivatives are bounded on $\Omega_l$.
 \item[\textbf{(A.5)}] The negative of the Hessian matrix of $\frac{1}{N_m}\log \mathcal{L}(\bm{\theta}_m|\mathcal{X}_m)$
 \begin{equation}
  \hat{\bm{H}}_m \triangleq -\frac{1}{N_m}\frac{d^2\log \mathcal{L}(\bm{\theta}_m|\mathcal{X}_m)}{d\bm{\theta}_md\bm{\theta}_m^\top}\bigg|_{\bm{\theta}_m=\hat{\bm{\theta}}_m} \in \mathbb{R}^{q\times q}
 \end{equation}
is positive definite, where $N_m$ is the cardinality of $\mathcal{X}_m$ $\left(N_m=\#\mathcal{X}_m\right)$. That is, $\min_{s,m}\lambda_s\left(\hat{\bm{H}}_m\right)>\epsilon$ for $s=1,\ldots,q$ and $m=1,\ldots,l$, where $\lambda_s\left(\hat{\bm{H}}_m\right)$ is the $s$th eigenvalue of $\hat{\bm{H}}_m$ and $\epsilon$ is a small positive constant.  
\end{enumerate}
The first step in Laplace's method of integration is to write the Taylor series expansion of $f(\bm{\theta}_m|M_l)$ and $\log\mathcal{L}(\bm{\theta}_m|\mathcal{X}_m)$ around $\hat{\bm{\theta}}_m, m=1,\ldots,l$. We begin by approximating $\log \mathcal{L}(\bm{\theta}_m|\mathcal{X}_m)$ by its second-order Taylor series expansion around $\hat{\bm{\theta}}_m$ as follows:
\begin{align}
  \log \mathcal{L}(\bm{\theta}_m|\mathcal{X}_m) &\!\approx\! \log \mathcal{L}(\hat{\bm{\theta}}_m|\mathcal{X}_m) \! +\!\tilde{\bm{\theta}}_m^\top\frac{d\log \mathcal{L}(\bm{\theta}_m|\mathcal{X}_m)}{d\bm{\theta}_m} \bigg|_{\bm{\theta}_m=\hat{\bm{\theta}}_m}\nonumber \\
  & + \frac{1}{2}\tilde{\bm{\theta}}_m^\top\left[\frac{d^2\log \mathcal{L}(\bm{\theta}_m|\mathcal{X}_m)}{d\bm{\theta}_m d\bm{\theta}_m^\top}\bigg|_{\bm{\theta}_m=\hat{\bm{\theta}}_m}\right]\tilde{\bm{\theta}}_m\nonumber \\
  & = \log \mathcal{L}(\hat{\bm{\theta}}_m|\mathcal{X}_m) - \frac{N_m}{2}\tilde{\bm{\theta}}_m^\top\hat{\bm{H}}_m\tilde{\bm{\theta}}_m,
\label{eq:loglikeli}
 \end{align}
where $\tilde{\bm{\theta}}_m\triangleq \bm{\theta}_m-\hat{\bm{\theta}}_m, m=1,\ldots,l$. The first derivative of $\log \mathcal{L}(\bm{\theta}_m|\mathcal{X}_m)$ evaluated at $\hat{\bm{\theta}}_m$ vanishes because of assumption \textbf{(A.3)}. With 
\begin{equation}
 U \triangleq \int_{\mathbb{R}^q} f(\bm{\theta}_m|M_l)\exp\left(\log\mathcal{L}(\bm{\theta}_m|\mathcal{X}_m)\right)d\bm{\theta}_m,
 \label{eq:U}
\end{equation}
substituting Eq.~\eqref{eq:loglikeli} into Eq.~\eqref{eq:U} and approximating $f(\bm{\theta}_m|M_l)$ by its Taylor series expansion yields
\begin{align}
U &\approx \int_{\mathbb{R}^q} \biggl(\biggl[f(\hat{\bm{\theta}}_m|M_l)+ \tilde{\bm{\theta}}_m^\top\frac{df(\bm{\theta}_m|M_l)}{d\bm{\theta}_m}\bigg|_{\bm{\theta}_m=\hat{\bm{\theta}}_m}+\text{HOT}\biggr] \nonumber \\
& \times \mathcal{L}(\hat{\bm{\theta}}_m|\mathcal{X}_m)\exp\left(-\frac{N_m}{2}\tilde{\bm{\theta}}_m^\top\hat{\bm{H}}_m\tilde{\bm{\theta}}_m\right)d\bm{\theta}_m\biggr),
\label{eq:U2}
\end{align}
where HOT denotes higher order terms and $\exp\left(-\frac{N_m}{2}\tilde{\bm{\theta}}_m^\top\hat{\bm{H}}_m\tilde{\bm{\theta}}_m\right)$ is a Gaussian kernel with mean $\hat{\bm{\theta}}_m$ and covariance matrix $\left(N_m\hat{\bm{H}}_m\right)^{-1}$. The second term in the first line of Eq.~\eqref{eq:U2} vanishes because it simplifies to $\kappa\E\left[\bm{\theta}_m-\hat{\bm{\theta}}_m\right]=0$, where $\kappa$ is a constant (see \cite[p. 53]{ando2010} for more details). Consequently, Eq.~\eqref{eq:U2} reduces to
\begin{align}
 U &\!\approx\! f(\hat{\bm{\theta}}_m|M_l)\mathcal{L}(\hat{\bm{\theta}}_m|\mathcal{X}_m) \int_{\mathbb{R}^q}\exp\left(-\frac{N_m}{2}\tilde{\bm{\theta}}_m^\top\hat{\bm{H}}_m\tilde{\bm{\theta}}_m\right)d\bm{\theta}_m \nonumber \\
& = f(\hat{\bm{\theta}}_m|M_l)\mathcal{L}(\hat{\bm{\theta}}_m|\mathcal{X}_m)\int_{\mathbb{R}^q}\biggl((2\pi)^{q/2}\left|N_m^{-1}\hat{\bm{H}}_m^{-1}\right|^{1/2} \nonumber \\
& \frac{1}{(2\pi)^{q/2}\left|N_m^{-1}\hat{\bm{H}}_m^{-1}\right|^{1/2}} \exp\left(-\frac{N_m}{2}\tilde{\bm{\theta}}_m^\top\hat{\bm{H}}_m\tilde{\bm{\theta}}_m\right)d\bm{\theta}_m\biggr) \nonumber \\
& = f(\hat{\bm{\theta}}_m|M_l)\mathcal{L}(\hat{\bm{\theta}}_m|\mathcal{X}_m)(2\pi)^{q/2}\left|N_m^{-1}\hat{\bm{H}}_m^{-1}\right|^{1/2},
\label{eq:U3}
\end{align}
where $|\cdot|$ stands for the determinant, given that $N_m\rightarrow\infty$. Using Eq.~\eqref{eq:U3}, we are thus able to provide an asymptotic approximation to the multidimensional integral in Eq.~\eqref{eq:logposterior}. Now, substituting Eq.~\eqref{eq:U3} into Eq.~\eqref{eq:logposterior}, we arrive at
\begin{align}
 \log p(M_l|\mathcal{X}) & \approx \log p(M_l) + \sum_{m=1}^l\log\left( f(\hat{\bm{\theta}}_m|M_l)\mathcal{L}(\hat{\bm{\theta}}_m|\mathcal{X}_m)\right) \nonumber \\
 &  + \frac{lq}{2}\log 2\pi - \frac{1}{2}\sum_{m=1}^l\log \left|\hat{\bm{J}}_m\right| + \rho, 
\label{eq:finalposterior}
\end{align}
where 
\begin{equation}
 \hat{\bm{J}}_m \triangleq N_m\hat{\bm{H}}_m = -\frac{d^2\log \mathcal{L}(\bm{\theta}_m|\mathcal{X}_m)}{d\bm{\theta}_md\bm{\theta}_m^\top}\bigg|_{\bm{\theta}_m=\hat{\bm{\theta}}_m} \in \mathbb{R}^{q\times q}
\end{equation}
is the Fisher Information Matrix (FIM) of data from the $m$th partition. \\
In the derivation of $\log p(M_l|\mathcal{X})$, so far, we have made no distributional assumption on the data set $\mathcal{X}$ except that the log-likelihood function $\log\mathcal{L}(\bm{\theta}_m|\mathcal{X}_m)$ and the prior on the parameter vectors $f(\bm{\theta}_m|M_l)$, for $m=1,\ldots,l$, should satisfy some mild conditions under each model $M_l\in\mathcal{M}$. Hence, Eq.~\eqref{eq:finalposterior} is a general expression of the posterior probability of the model $M_l$ given $\mathcal{X}$ for a general class of data distributions that satisfy assumptions \textbf{(A.2)}-\textbf{(A.5)}. The BIC is concerned with the computation of the posterior probability of candidate models and thus Eq.~\eqref{eq:finalposterior} can also be written as
\begin{align}
 \text{BIC}_{\mbox{\tiny G}}(M_l) &\triangleq \log p(M_l|\mathcal{X}) \nonumber \\
 & \approx \log p(M_l) + \log f(\hat{\bm{\Theta}}_l|M_l) + \log \mathcal{L}(\hat{\bm{\Theta}}_l|\mathcal{X}) \nonumber \\
 & + \frac{lq}{2}\log 2\pi - \frac{1}{2}\sum_{m=1}^l\log \left|\hat{\bm{J}}_m\right| + \rho. 
\label{eq:BICgen}
 \end{align}
After calculating $\text{BIC}_{\mbox{\tiny G}}(M_l)$ for each candidate model $M_l\in\mathcal{M}$, the number of clusters in $\mathcal{X}$ is estimated as
\begin{equation}
 \hat{K}_{\text{BIC}_{\mbox{\tiny G}}} = \underset{l=L_\mathrm{min},\ldots,L_\mathrm{max}}{\arg \max}\text{BIC}_{\mbox{\tiny G}}(M_l).
 \label{eq:BICGGK}
\end{equation}
However, calculating $\text{BIC}_{\mbox{\tiny G}}(M_l)$ using Eq.~\eqref{eq:BICgen} is a computationally expensive task as it requires the estimation of the FIM, $\hat{\bm{J}}_m$, for each cluster $m=1,\ldots,l$ in the candidate model $M_l\in\mathcal{M}$. Our objective is to find an asymptotic approximation for $\log \left|\hat{\bm{J}}_m\right|, m=1,\ldots,l,$ in order to simplify the computation of $\text{BIC}_{\mbox{\tiny G}}(M_l)$. We solve this problem by imposing specific assumptions on the distribution of the data set $\mathcal{X}$. In the next section, we provide an asymptotic approximation for $\log\left|\hat{\bm{J}}_m\right|, m=1,\ldots,l,$ assuming that $\mathcal{X}_m$ contains iid multivariate Gaussian data points.  

%%%%%%%%%%%%%%%%%%%%%%%%%%%%%%%%%%%%%%%%%%%%%%%%%%%%%%%%%%%%%%%%%%%%%%%
\section{Proposed Bayesian Cluster Enumeration Algorithm for Multivariate Gaussian Data}
\label{sec:gauss}
We propose a two-step approach to estimate the number of partitions (clusters) in $\mathcal{X}$ and provide an estimate of cluster parameters, such as cluster centroids and covariance matrices, in an unsupervised learning framework. The proposed approach consists of a model-based clustering algorithm, which clusters the data set $\mathcal{X}$ according to each candidate model $M_l\in\mathcal{M}$, and a Bayesian cluster enumeration criterion that selects the model which is a posteriori most probable. 

\subsection{Proposed Bayesian Cluster Enumeration Criterion for Multivariate Gaussian Data}
Let $\mathcal{X}\triangleq\{\bm{x}_1,\ldots,\bm{x}_N\}$ denote the observed data set which can be partitioned into $K$ clusters $\{\mathcal{X}_1,\ldots,\mathcal{X}_K\}$. Each cluster $\mathcal{X}_k, k\in\mathcal{K},$ contains $N_k$ data vectors that are realizations of iid Gaussian random variables $\bm{x}_k\sim\mathcal{N}(\bm{\mu}_k,\bm{\Sigma}_k)$, where $\bm{\mu}_k\in\mathbb{R}^{r\times 1}$ and $\bm{\Sigma}_k\in\mathbb{R}^{r\times r}$ represent the centroid and the covariance matrix of the $k$th cluster, respectively. Further, let $\mathcal{M}\triangleq\{M_{L_\mathrm{min}},\ldots,M_{L_\mathrm{max}}\}$ denote a set of Gaussian candidate models and let there be a clustering algorithm that partitions $\mathcal{X}$ into $l$ independent, mutually exclusive, and non-empty subsets (clusters) $\mathcal{X}_m, m=1,\ldots,l,$ by providing parameter estimates $\hat{\bm{\theta}}_m =[\hat{\bm{\mu}}_m,\hat{\bm{\Sigma}}_m]^\top$ for each candidate model $M_l\in\mathcal{M}$, where $l=L_\mathrm{min},\ldots,L_\mathrm{max}$ and $l\in\mathbb{Z}^{+}$. Assume that \textbf{(A.1)}-\textbf{(A.7)} are satisfied. 
\begin{theorem}
The posterior probability of $M_l\in\mathcal{M}$ given $\mathcal{X}$ can be asymptotically approximated as
 \begin{align}
  \text{\emph{BIC}}_{\mbox{\tiny \emph{N}}}(M_l) &\triangleq \log p(M_l|\mathcal{X}) \nonumber \\
  & \approx \sum_{m=1}^l N_m\log N_m - \sum_{m=1}^l\frac{N_m}{2}\log \left|\hat{\bm{\Sigma}}_m\right| \nonumber \\
  & - \frac{q}{2}\sum_{m=1}^l\log N_m,
  \label{eq:BICG}
 \end{align}
where $N_m$ is the cardinality of the subset $\mathcal{X}_m$ and it satisfies $N=\sum_{m=1}^lN_m$. The term $\sum_{m=1}^l\log N_m$ sums the logarithms of the number of data vectors in each cluster $m=1,\ldots,l$.
\end{theorem}
\begin{proof}
Proving Theorem 1 requires finding an asymptotic approximation to $\log\left|\hat{\bm{J}}_m\right|$ in Eq.~\eqref{eq:BICgen} and, based on this approximation, deriving an expression for $\text{BIC}_{\mbox{\tiny N}}(M_l)$. A detailed proof is given in Appendix~\ref{app:A}.
\end{proof}
\noindent Once the Bayesian Information Criterion, $\text{{BIC}}_{\mbox{\tiny {N}}}(M_l)$, is computed for each candidate model $M_l\in\mathcal{M}$, the number of partitions (clusters) in $\mathcal{X}$ is estimated as
\begin{equation}
 \hat{K}_{\text{{BIC}}_{\mbox{\tiny {N}}}} = \underset{l=L_\mathrm{min},\ldots,L_\mathrm{max}}{\arg \max}\text{{BIC}}_{\mbox{\tiny {N}}}(M_l).
 \label{eq:BICGK}
\end{equation}
\begin{remark}
The proposed criterion, $\text{\emph {BIC}}_{\mbox{\tiny \emph N}}$, and the original BIC as derived in  \cite{schwarz1978,cavanaugh1999} differ in terms of their penalty terms. A detailed discussion is provided in Section \ref{sec:penTerm}.
\end{remark}
\noindent The first step in calculating $\text{BIC}_{\mbox{\tiny N}}(M_l)$ for each model $M_l\in\mathcal{M}$ is the partitioning of the data set $\mathcal{X}$ into $l$ clusters $\mathcal{X}_m, m=1,\ldots,l,$ and the estimation of the associated cluster parameters using an unsupervised learning algorithm. Since the approximations in $\text{BIC}_{\mbox{\tiny N}}(M_l)$ are based on maximizing the likelihood function of Gaussian distributed random variables, we use a clustering algorithm that is based on the maximum likelihood principle. Accordingly, a natural choice is the EM algorithm for Gaussian mixture models. 

\subsection{The Expectation-Maximization (EM) Algorithm for Gaussian Mixture Models}
\label{subsec:EM}
The EM algorithm finds maximum likelihood solutions for models with latent variables \cite{bishop2006}. In our case, the latent variables are the cluster memberships of the data vectors in $\mathcal{X}$, given that the $l$-component Gaussian mixture distribution of a data vector $\bm{x}_n$ can be written as
\begin{equation}
 f(\bm{x}_n|M_l,\bm{\Theta}_l) = \sum_{m=1}^l\tau_m g(\bm{x}_n;\bm{\mu}_m,\bm{\Sigma}_m),
\end{equation}
where $g(\bm{x}_n;\bm{\mu}_m,\bm{\Sigma}_m)$ represents the $r$-variate Gaussian pdf and $\tau_m$ is the mixing coefficient of the $m$th cluster. The goal of the EM algorithm is to maximize the log-likelihood function of the data set $\mathcal{X}$ with respect to the parameters of interest as follows:
\begin{equation}
\arg \underset{\bm{\Psi}_l}{\max}\log \mathcal{L}(\bm{\Psi}_l|\mathcal{X})\!\!=\!\! \arg \underset{\bm{\Psi}_l}{\max}\!\!\sum_{n=1}^N\!\log \!\! \sum_{m=1}^l\!\tau_m g(\bm{x}_n;\bm{\mu}_m,\bm{\Sigma}_m),
\label{eq:maxln}
\end{equation}
where $\bm{\Psi}_l=[\bm{\tau}_l,\bm{\Theta}_l^\top]$ and $\bm{\tau}_l=\left[\tau_1,\ldots,\tau_l\right]^\top$. Maximizing Eq.~\eqref{eq:maxln} with respect to the elements of $\bm{\Psi}_l$ results in coupled equations. The EM algorithm solves these coupled equations using a two-step iterative procedure. The first step (E step) evaluates $\hat{\upsilon}^{(i)}_{nm}$, which is an estimate of the probability that data vector $\bm{x}_n$ belongs to the $m$th cluster at the $i$th iteration, for $n=1,\ldots,N$ and $m=1,\ldots,l$. $\hat{\upsilon}_{nm}^{(i)}$ is calculated as follows:
\begin{equation}
  \hat{\upsilon}_{nm}^{(i)} = \frac{\hat{\tau}_m^{(i-1)}g(\bm{x}_n;\hat{\bm{\mu}}_m^{(i-1)},\hat{\bm{\Sigma}}_m^{(i-1)})}{\sum_{j=1}^l\hat{\tau}_j^{(i-1)}g(\bm{x}_n;\hat{\bm{\mu}}_j^{(i-1)},\hat{\bm{\Sigma}}_j^{(i-1)})},
\label{eq:upsilon}
\end{equation}
where $\hat{\bm{\mu}}_m^{(i-1)}$ and $\hat{\bm{\Sigma}}_m^{(i-1)}$ represent the centroid and covariance matrix estimates, respectively, of the $m$th cluster at the previous iteration ($i-1$). The second step (M step) re-estimates the cluster parameters using the current values of $\hat{\upsilon}_{nm}$ as follows:
 \begin{align}
  \hat{\bm{\mu}}_m^{(i)} &= \frac{\sum_{n=1}^N\hat{\upsilon}_{nm}^{(i)}\bm{x}_n}{\sum_{n=1}^N\hat{\upsilon}_{nm}^{(i)}} \label{eq:muhat} \\
  \hat{\bm{\Sigma}}_m^{(i)} &= \frac{\sum_{n=1}^N\hat{\upsilon}_{nm}^{(i)}(\bm{x}_n-\hat{\bm{\mu}}_m^{(i)})(\bm{x}_n-\hat{\bm{\mu}}_m^{(i)})^\top}{\sum_{n=1}^N\hat{\upsilon}_{nm}^{(i)}} \label{eq:sigmahat} \\
  \hat{\tau}_m^{(i)} &= \frac{\sum_{n=1}^N\hat{\upsilon}_{nm}^{(i)}}{N} \label{eq:tauhat}
 \end{align}
The E and M steps are performed iteratively until either the cluster parameter estimates $\hat{\bm{\Psi}}_l$ or the log-likelihood estimate $\log \mathcal{L}(\hat{\bm{\Psi}}_l|\mathcal{X})$ converges. \\
A summary of the estimation of the number of clusters in an observed data set using the proposed two-step approach is provided in Algorithm~\ref{alg:BICG}. Note that the computational complexity of $\text{BIC}_{\mbox{\tiny N}}(M_l)$ is only $\mathcal{O}(1)$, which can easily be ignored during the run-time analysis of the proposed approach.  Hence, since the EM algorithm is run for all candidate models in $\mathcal{M}$, the computational complexity of the proposed two-step approach is $\mathcal{O}(\zeta Nr^2\left(L_{\mathrm{min}}+\ldots+L_{\mathrm{max}}\right))$, where $\zeta$ is a fixed stopping threshold of the EM algorithm. 
\begin{algorithm}[htb]
 \caption{Proposed two-step cluster enumeration approach}
 \begin{algorithmic}
  \State \textit{Inputs:} data set $\mathcal{X}$; set of candidate models $\mathcal{M}\triangleq \{M_{L_\mathrm{min}},\ldots,M_{L_\mathrm{max}}\}$ 
  \For {$l=L_\mathrm{min},\ldots,L_\mathrm{max}$}
  \State {\it Step 1:} Model-based clustering
  \State {\it Step 1.1:} The EM algorithm
  \For {$m=1,\ldots,l$}
  \State Initialize $\bm{\mu}_m$ using K-means++ \cite{blomer2016,arthur2007}
  \State $\hat{\bm{\Sigma}}_m = \frac{1}{N_m}\sum_{\bm{x}_n\in\mathcal{X}_m}(\bm{x}_{n}-\hat{\bm{\mu}}_m)(\bm{x}_{n}-\hat{\bm{\mu}}_m)^\top$
  \State $\hat{\tau}_m = \frac{N_m}{N}$
 \EndFor
 \For {$i=1,2,\ldots$}
 \State \textit{E step:}
 \For {$n=1,\ldots,N$}
 \For {$m=1,\ldots,l$}
 \State Calculate $\hat{\upsilon}_{nm}^{(i)}$ using Eq.~\eqref{eq:upsilon}
 \EndFor
 \EndFor
 \State \textit{M step:}
 \For {$m=1,\ldots,l$}
 \State Determine $\hat{\bm{\mu}}_m^{(i)}$, $\hat{\bm{\Sigma}}_m^{(i)}$, and $\hat{\tau}_m^{(i)}$ via Eqs.~(\ref{eq:muhat})-(\ref{eq:tauhat})
 \EndFor
 \State Check for convergence of either $\hat{\bm{\Psi}}_l^{(i)}$ or $\log \mathcal{L}(\hat{\bm{\Psi}}_l^{(i)}|\mathcal{X})$
 \If {convergence condition is satisfied}
 \State Exit for loop
 \EndIf
 \EndFor
 \State \textit{Step 1.2:} Hard clustering
 \For {$n=1,\ldots,N$}
 \For {$m=1,\ldots,l$}
  \begin{equation*}
  \iota_{nm} = \begin{cases}
                               1, & m=\underset{j=1,\ldots,l}{\arg\max}\!\!\!\!\quad\hat{\upsilon}_{nj}^{(i)} \\
                               0, & \text{otherwise}
                      \end{cases}
 \end{equation*}
 \EndFor
 \EndFor
 \For {$m=1,\ldots,l$}
 \State $N_m = \sum_{n=1}^N\iota_{nm}$ 
 \EndFor
 \State {\it Step 2:} Calculate $\text{BIC}_{\mbox{\tiny N}}(M_l)$ via Eq.~\eqref{eq:BICG}
 \EndFor
 \State Estimate the number of clusters, $\hat{K}_{\text{BIC}_{\mbox{\tiny N}}}$, in $\mathcal{X}$ via Eq.~\eqref{eq:BICGK}
 \end{algorithmic}
 \label{alg:BICG}
\end{algorithm}

%%%%%%%%%%%%%%%%%%%%%%%%%%%%%%%%%%%%%%%%%%%%%%%%%%%%%%%%%%%%%%%%%%%%%%%
\section{Existing BIC-Based Cluster Enumeration Methods}
\label{sec:stateoftheart}
As discussed in Section~\ref{sec:intro}, existing cluster enumeration algorithms that are based on the original BIC use the criterion as it is known from parameter estimation tasks without questioning its validity on cluster analysis. Nevertheless, since these criteria have been widely used, we briefly review them to provide a comparison to the proposed criterion $\text{BIC}_{\mbox{\tiny N}}$, which is given by Eq.~\eqref{eq:BICG}.  \\
The original BIC, as derived in \cite{cavanaugh1999}, evaluated at a candidate model $M_l\in\mathcal{M}$ is written as
\begin{equation}
 \text{BIC}_{\mbox{\tiny O}}(M_l) = 2\log \mathcal{L}(\hat{\bm{\Theta}}_l|\mathcal{X}) - ql\log N,
 \label{eq:BICc}
\end{equation}
where $\mathcal{L}(\hat{\bm{\Theta}}_l|\mathcal{X})$ denotes the likelihood function and $N=\#\mathcal{X}$. In Eq.~\eqref{eq:BICc}, $2\log \mathcal{L}(\hat{\bm{\Theta}}_l|\mathcal{X})$ denotes the data-fidelity term, while $ql\log N$ is the penalty term. Under the assumption that the observed data is Gaussian distributed, the data-fidelity terms of $\text{BIC}_{\mbox{\tiny O}}$ and the ones of our proposed criterion, $\text{BIC}_{\mbox{\tiny N}}$, are exactly the same. The only deference between the two is the penalty term. Hence, we use a similar procedure as in Algorithm~\ref{alg:BICG} to implement the original BIC as a wrapper around the EM algorithm. \\
Moreover, the original BIC is commonly used as a wrapper around K-means by assuming that the data points that belong to each cluster are iid as Gaussian and all clusters are spherical with an identical variance, i.e. $\bm{\Sigma}_m = \bm{\Sigma}_j=\sigma^2\bm{I}_r$ for $m\neq j$, where $\sigma^2$ is the common variance of the clusters in $M_l$ \cite{pelleg2000,zhao2008,zhao22008}. Under these assumptions, the original BIC is given by
\begin{equation}
 \text{BIC}_{\mbox{\tiny OS}}(M_l) = 2\log \mathcal{L}(\hat{\bm{\Theta}}_l|\mathcal{X}) - \alpha\log N,
 \label{eq:BICcs1}
\end{equation}
where $\text{BIC}_{\mbox{\tiny OS}}(M_l)$ denotes the original BIC of the candidate model $M_l$ derived under the assumptions stated above and $\alpha=\left(rl+1\right)$ is the number of estimated parameters in $M_l\in\mathcal{M}$. Ignoring the model independent terms, $\text{BIC}_{\mbox{\tiny OS}}(M_l)$ can be written as 
\begin{equation}
 \text{BIC}_{\mbox{\tiny OS}}(M_l) = 2\sum_{m=1}^lN_m\log N_m - rN\log\hat{\sigma}^2 - \alpha\log N, 
 \label{eq:BICcs}
\end{equation}
where 
\begin{equation}
 \hat{\sigma}^2 = \frac{1}{rN}\sum_{m=1}^l\sum_{\bm{x}_n\in\mathcal{X}_m}\left(\bm{x}_{n}-\hat{\bm{\mu}}_m\right)^\top\left(\bm{x}_{n}-\hat{\bm{\mu}}_m\right)
 \label{eq:sigma}
\end{equation}
is the maximum likelihood estimator of the common variance. \\
In our experiments, we implement $\text{BIC}_{\mbox{\tiny OS}}$ as a wrapper around the K-means++ algorithm \cite{arthur2007}. The implementation of the proposed BIC as a wrapper around the K-means++ algorithm is given by Eq.~\eqref{eq:BICNS}.

%%%%%%%%%%%%%%%%%%%%%%%%%%%%%%%%%%%%%%%%%%%%%%%%%%%%%%%%%%%%%%%%%%%%%%%%%
\section{Comparison of the Penalty Terms of Different Bayesian Cluster Enumeration Criteria}
\label{sec:penTerm}
Comparing Eqs.~\eqref{eq:BICG}, ~\eqref{eq:BICc}, and ~\eqref{eq:BICcs1}, we notice that they have a common form \cite{stoica2004,rao1989}, that is, 
\begin{equation}
 2\log \mathcal{L}(\hat{\bm{\Theta}}_l|\mathcal{X}) - \eta,
 \label{eq:BICuni}
\end{equation}
but with different penalty terms, where
\begin{align}
 \text{BIC}_{\mbox{\tiny N}}:& \quad \eta = q\sum_{m=1}^l\log N_m \label{eq:penN} \\
 \text{BIC}_{\mbox{\tiny O}}:& \quad \eta = ql\log N \label{eq:penC} \\
 \text{BIC}_{\mbox{\tiny OS}}:& \quad \eta = \left(rl+1\right)\log N. \label{eq:penCS} 
\end{align}
\begin{remark}
 $\text{\emph{BIC}}_{\mbox{\tiny \emph{O}}}$ and $\text{\emph{BIC}}_{\mbox{\tiny \emph{OS}}}$ carry information about the structure of the data only on their data-fidelity term, which is the first term in Eq.~\eqref{eq:BICuni}. On the other hand, as shown in Eq.~\eqref{eq:BICG}, both the data-fidelity and penalty terms of our proposed criterion, $\text{\emph{BIC}}_{\mbox{\tiny \emph{N}}}$, contain information about the structure of the data.  
\end{remark}
\noindent The penalty terms of $\text{BIC}_{\mbox{\tiny O}}$ and $\text{BIC}_{\mbox{\tiny OS}}$ depend linearly on $l$, while the penalty term of our proposed criterion, $\text{BIC}_{\mbox{\tiny N}}$, depends on $l$ in a non-linear manner. Comparing the penalty terms in Eqs.~\eqref{eq:penN}-\eqref{eq:penCS}, $\text{BIC}_{\mbox{\tiny OS}}$ has the weakest penalty term. In the asymptotic regime, the penalty terms of $\text{BIC}_{\mbox{\tiny N}}$ and $\text{BIC}_{\mbox{\tiny O}}$ coincide. But, in the finite sample regime, for values of $l>1$, the penalty term of $\text{BIC}_{\mbox{\tiny O}}$ is stronger than the penalty term of $\text{BIC}_{\mbox{\tiny N}}$. Note that the penalty term of our proposed criterion, $\text{BIC}_{\mbox{\tiny N}}$, depends on the number of data vectors in each cluster, $N_m, m=1,\ldots,l$, of each candidate model $M_l\in\mathcal{M}$, while the penalty term of the original BIC depends only on the total number of data vectors in the data set. Hence, the penalty term of our proposed criterion might exhibit sensitivities to the initialization of cluster parameters and the associated number of data vectors per cluster.    

%%%%%%%%%%%%%%%%%%%%%%%%%%%%%%%%%%
\section{Experimental Evaluation}
\label{sec:results}
In this section, we compare the cluster enumeration performance of our proposed criterion, $\text{BIC}_{\mbox{\tiny N}}$ given by Eq.~\eqref{eq:BICG}, with the cluster enumeration methods discussed in Section~\ref{sec:stateoftheart}, namely $\text{BIC}_{\mbox{\tiny O}}$ and $\text{BIC}_{\mbox{\tiny OS}}$ given by Eqs.~\eqref{eq:BICc} and \eqref{eq:BICcs}, respectively, using synthetic and real data sets. We first describe the performance measures used for comparing the different cluster enumeration criteria. Then, the numerical experiments performed on synthetic data sets and the results obtained from real data sets are discussed in detail. For all simulations, we assume that a family of candidate models $\mathcal{M}\triangleq\{M_{L_{\mathrm{min}}},\ldots,M_{L_{\mathrm{max}}}\}$ is given with $L_{\mathrm{min}}=1$ and $L_{\mathrm{max}}=2K$, where $K$ is the true number of clusters in the data set $\mathcal{X}$. All simulation results are an average of $1000$ Monte Carlo experiments unless stated otherwise. The compared cluster enumeration criteria are based on the same initial cluster parameters in each Monte Carlo experiment, which allows for a fair comparison. \\
The MATLAB\textsuperscript{\textcopyright} code that implements the proposed two-step algorithm and the Bayesian cluster enumeration methods discussed in Section \ref{sec:stateoftheart} is available at:\\ \href{https://github.com/FreTekle/Bayesian-Cluster-Enumeration}{\color{blue}https://github.com/FreTekle/Bayesian-Cluster-Enumeration}

\subsection{Performance Measures}
\label{subsec:perf}
The empirical probability of detection $\left(p_{\text{det}}\right)$, the empirical probability of underestimation $\left(p_{\text{under}}\right)$, the empirical probability of selection, and the Mean Absolute Error (MAE) are used as performance measures. The empirical probability of detection is defined as the probability with which the correct number of clusters is selected and it is calculated as
\begin{equation}
 p_{\text{det}} =\frac{1}{\text{MC}} \sum_{i=1}^{\text{MC}}\mathbbm{1}_{\{\hat{K}_i=K\}},
\end{equation}
where $\text{MC}$ is the total number of Monte Carlo experiments, $\hat{K}_i$ is the estimated number of clusters in the $i$th Monte Carlo experiment, and $\mathbbm{1}_{\{\hat{K}_i=K\}}$ is an indicator function which is defined as
\begin{equation}
 \mathbbm{1}_{\{\hat{K}_i=K\}} \triangleq 
 \begin{cases}
  1, \quad \textrm{if} \quad \hat{K}_i=K \\
  0, \quad \text{otherwise}
 \end{cases}.
\end{equation}
The empirical probability of underestimation $\left(p_{\text{under}}\right)$ is the probability that $\hat{K}<K$. The empirical probability of selection is defined as the probability with which the number of clusters specified by each candidate model in $\mathcal{M}$ is selected. The last performance measure, which is the Mean Absolute Error (MAE), is computed as 
\begin{equation}
 \textrm{MAE} = \frac{1}{\text{MC}}\sum_{i=1}^{\text{MC}}\left|K-\hat{K}_i\right|.
\end{equation}

\subsection{Numerical Experiments}
\label{subsec:numexp}
\subsubsection{Simulation Setup}
We consider two synthetic data sets, namely Data-1 and Data-2, in our simulations. Data-1, shown in Fig.~\ref{fig:data3}, contains realizations of the random variables $\bm{x}_k\sim\mathcal{N}(\bm{\mu}_k,\bm{\Sigma}_k)$, where $k = 1,2,3$, with cluster centroids $\bm{\mu}_1=[2, 3.5]^\top, \bm{\mu}_2=[6, 2.7]^\top, \bm{\mu}_3=[9, 4]^\top,$ and covariance matrices 
\[
 \bm{\Sigma}_1 \!=\! 
 \begin{bmatrix}
  0.2 & 0.1 \\
  0.1 & 0.75
 \end{bmatrix}\!, 
 \bm{\Sigma}_2 \!=\! 
 \begin{bmatrix}
  0.5 & 0.25 \\
  0.25 & 0.5
 \end{bmatrix}\!, 
 \bm{\Sigma}_3 \!=\! 
 \begin{bmatrix}
  1 & 0.5 \\
  0.5 & 1
 \end{bmatrix}\!.
\]
The first cluster is linearly separable from the others, while the remaining clusters are overlapping. The number of data vectors per cluster is specified as $N_1=\gamma\times50$, $N_2=\gamma\times100$, and $N_3=\gamma\times200$, where $\gamma$ is a constant. \\
The second data set, Data-2, contains realizations of the random variables $\bm{x}_k\sim\mathcal{N}(\bm{\mu}_k,\bm{\Sigma}_k)$, where $k = 1,\ldots,10$, with cluster centroids $\bm{\mu}_1=[0,0]^\top$, $\bm{\mu}_{2}=[3,-2.5]$, $\bm{\mu}_3=[3,1]^\top$, $\bm{\mu}_4=[-1,-3]^\top$, $\bm{\mu}_5=[-4,0]^\top$, $\bm{\mu}_6=[-1,1]^\top$, $\bm{\mu}_7=[-3,3]^\top$, $\bm{\mu}_8=[2.5,4]^\top$, $\bm{\mu}_9=[-3.5,-2.5]$, $\bm{\mu}_{10}=[0,3]^\top$, and covariance matrices
\[
\bm{\Sigma}_1\!=\! 
 \begin{bmatrix}
  0.25 & -0.15 \\
  -0.15 & 0.15
 \end{bmatrix}\!,
 \bm{\Sigma}_{2}\!=\! 
 \begin{bmatrix}
  0.5 & 0 \\
  0 & 0.15
 \end{bmatrix}\!,
 \bm{\Sigma}_i\!=\! 
 \begin{bmatrix}
  0.1 & 0 \\
  0 & 0.1
 \end{bmatrix}\!,
\]
where $i=3,\ldots,10$. As depicted in Fig.~\ref{fig:data10}, Data-2 contains eight identical and spherical clusters and two elliptical clusters. There exists an overlap between two clusters, while the rest of the clusters are well separated. All clusters in this data set have the same number of data vectors. 
\begin{figure}[t]
 \centering
 \begin{subfigure}[b]{\linewidth}
 \centering
  \includegraphics[width=\linewidth]{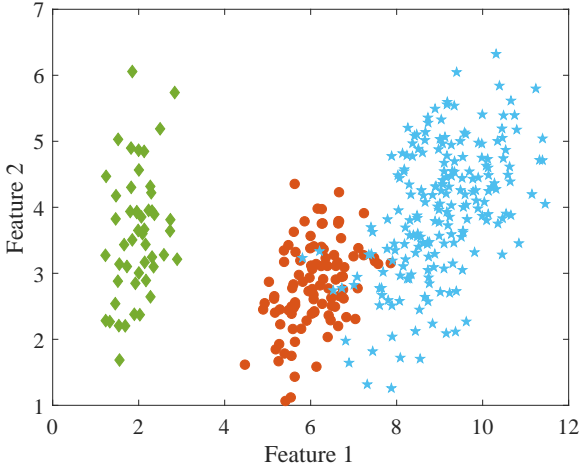}
 \caption{Data-1 for $\gamma=1$}
 \label{fig:data3}
 \end{subfigure}
\begin{subfigure}[b]{\linewidth}
\centering
 \includegraphics[width=\linewidth]{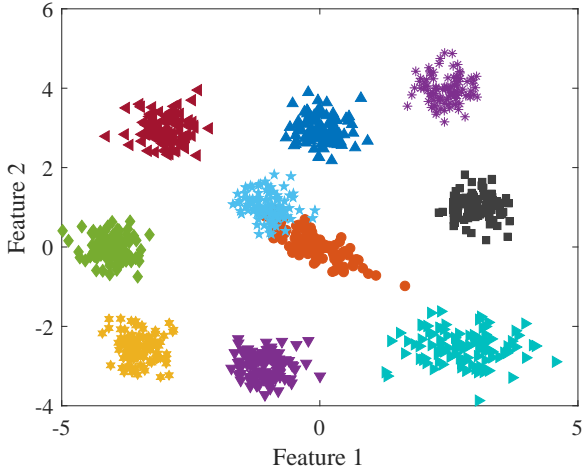}
 \caption{Data-2 for $N_k=100$}
 \label{fig:data10}
\end{subfigure}
\caption{Synthetic data sets.}
\end{figure}

\subsubsection{Simulation Results}
Data-1 is particularly challenging for cluster enumeration criteria because it has not only overlapping but also unbalanced clusters. Cluster unbalance refers to the fact that different clusters have a different number of data vectors, which might result in some clusters dominating the others.  The impact of cluster overlap and unbalance on $p_{\text{det}}$ and MAE is displayed in Table~\ref{tab:data2P+mae}. This table shows $p_{\text{det}}$ and MAE as a function of $\gamma$, where $\gamma$ is allowed to take values from the set $\{1,3,6,12,48\}$. The cluster enumeration performance of $\text{BIC}_{\mbox{\tiny OS}}$ is lower than the other methods because it is designed for spherical clusters with identical variance, while Data-1 has one elliptical and two spherical clusters with different covariance matrices. Our proposed criterion performs best in terms of $p_\text{det}$ and MAE for all values of $\gamma$. As $\gamma$ increases, which corresponds to an increase in the number of data vectors in the data set, the cluster enumeration performance of $\text{BIC}_{\mbox{\tiny N}}$ and $\text{BIC}_{\mbox{\tiny O}}$ greatly improves, while the performance of $\text{BIC}_{\mbox{\tiny OS}}$ deteriorates because of the increase in overestimation. The total criterion (BIC) and penalty term of different Bayesian cluster enumeration criteria as a function of the number of clusters specified by the candidate models for $\gamma=1$ is depicted in Fig.~\ref{fig:data3_gen}. The BIC plot in Fig.~\ref{fig:BIC} is the result of one Monte Carlo run. It shows that $\text{BIC}_{\mbox{\tiny N}}$ and $\text{BIC}_{\mbox{\tiny O}}$ have a maximum at the true number of clusters $(K=3)$, while $\text{BIC}_{\mbox{\tiny OS}}$ overestimates the number of clusters to $\hat{K}_{\text{BIC}_{\mbox{\tiny OS}}}=4$. As shown in Fig.~\ref{fig:data2Pen}, our proposed criterion, $\text{BIC}_{\mbox{\tiny N}}$, has the second strongest penalty term. Note that, the penalty term of our proposed criterion shows a curvature at the true number of clusters, while the penalty terms of $\text{BIC}_{\mbox{\tiny O}}$ and $\text{BIC}_{\mbox{\tiny OS}}$ are uninformative on their own. \\
%%%% Data-1
\begin{table}[t]
 \caption{The empirical probability of detection in $\%$, the empirical probability of underestimation in $\%$, and the Mean Absolute Error (MAE) of various Bayesian cluster enumeration criteria as a function of $\gamma$ for Data-1.}
 \centering
 \begin{tabular}{ccccccc}
  \toprule
  $\gamma$ & & $1$ & $3$ & $6$ & $12$ & $48$ \\
  \midrule
  \multirow{2}{*}{$p_{\text{det}}(\%)$} & $\text{BIC}_{\mbox{\tiny N}}$ & $\bm{55.2}$ & $\bm{74.3}$ & $\bm{87.4}$ & $\bm{95.7}$ & $\bm{100}$ \\
  & $\text{BIC}_{\mbox{\tiny O}}$ & $43.6$ & $69.7$ & $85.1$ & $94.9$ &  $\bm{100}$ \\
  & $\text{BIC}_{\mbox{\tiny OS}}$ & $53.9$ & $50.5$ & $49.4$ & $42.4$ &  $31.8$ \\
  \midrule
  \multirow{2}{*}{$p_{\text{under}}(\%)$} & $\text{BIC}_{\mbox{\tiny N}}$ & $44.5$ & $25.7$ & $12.6$ & $4.3$ & $0$ \\
  & $\text{BIC}_{\mbox{\tiny O}}$ & $56.4$ & $30.3$ & $14.9$ & $5.1$ & $0$ \\
  & $\text{BIC}_{\mbox{\tiny OS}}$ & $0$ & $0$ & $0$ & $0$ & $0$ \\
  \midrule
  \multirow{2}{*}{$\text{MAE}$} & $\text{BIC}_{\mbox{\tiny N}}$ & $\bm{0.449}$ & $\bm{0.257}$ & $\bm{0.126}$ & $\bm{0.043}$ & $\bm{0}$ \\
  & $\text{BIC}_{\mbox{\tiny O}}$ & $0.564$ & $0.303$ & $0.149$ & $0.051$ & $\bm{0}$ \\
  & $\text{BIC}_{\mbox{\tiny OS}}$ & $0.469$ & $0.495$ & $0.506$ & $0.576$ &  $0.682$ \\
  \bottomrule
 \end{tabular}
 \label{tab:data2P+mae}
\end{table}
%%%% Data-3
\begin{table}[t]
 \caption{The empirical probability of detection in $\%$, the empirical probability of underestimation in $\%$, and the Mean Absolute Error (MAE) of various Bayesian cluster enumeration criteria as a function of the number of data vectors per cluster $(N_k)$ for Data-2.}
 \centering
 \begin{tabular}{cccccc}
  \toprule
   $N_k$ & &  $100$ & $200$ & $500$ & $1000$ \\
  \midrule
  \multirow{2}{*}{$p_{\text{det}}(\%)$} & $\text{BIC}_{\mbox{\tiny N}}$ &  $\bm{56.1}$ & $\bm{66}$ & $\bm{81}$ & $\bm{85.3}$ \\
  & $\text{BIC}_{\mbox{\tiny O}}$ &  $41$ & $57.1$ & $78$ & $84.9$  \\
  & $\text{BIC}_{\mbox{\tiny OS}}$ &  $2.7$ & $0.9$ & $0.1$ & $0$  \\
  \midrule
  \multirow{2}{*}{$p_{\text{under}}(\%)$} & $\text{BIC}_{\mbox{\tiny N}}$ &  $37.6$ & $30.2$ & $18.2$ & $13.5$ \\
  & $\text{BIC}_{\mbox{\tiny O}}$ &  $58.6$ & $41.7$ & $21.4$ & $14.1$ \\
  & $\text{BIC}_{\mbox{\tiny OS}}$ &  $0$ & $0$ & $0$ & $0$ \\
  \midrule
  \multirow{2}{*}{$\text{MAE}$} & $\text{BIC}_{\mbox{\tiny N}}$ &  $\bm{0.452}$ & $\bm{0.341}$ & $\bm{0.19}$ & $\bm{0.148}$ \\
  & $\text{BIC}_{\mbox{\tiny O}}$ &  $0.59$ & $0.429$ & $0.22$ & $0.151$ \\
  & $\text{BIC}_{\mbox{\tiny OS}}$ &  $1.613$ & $1.659$ & $1.745$ & $1.8$ \\
  \bottomrule
 \end{tabular}
 \label{tab:Data-2}
\end{table}
Table \ref{tab:Data-2} shows $p_\text{det}$ and MAE as a function of the number of data vectors per cluster, $N_k, k=1,\ldots,10$, where $N_k$ is allowed to take values from the set $\{100,200,500,1000\}$, for Data-2. Data-2 contains both spherical and elliptical clusters and there is an overlap between two clusters, while the rest of the clusters are well separated. The proposed criterion, $\text{BIC}_{\mbox{\tiny N}}$, consistently outperforms the cluster enumeration methods that are based on the original BIC for the specified number of data vectors per cluster $(N_k)$. $\text{BIC}_{\mbox{\tiny O}}$ tends to underestimate the number of clusters to $\hat{K}_{\text{BIC}_{\mbox{\tiny O}}}=9$ when $N_k$ is small, and it merges the two overlapping clusters. Even though majority of the clusters are spherical, $\text{BIC}_{\mbox{\tiny OS}}$ rarely finds the correct number of clusters. \\
%%%% Data-1
\begin{figure}[t]
 \centering
 \begin{subfigure}[b]{\linewidth}
 \centering
  \includegraphics[width=\linewidth]{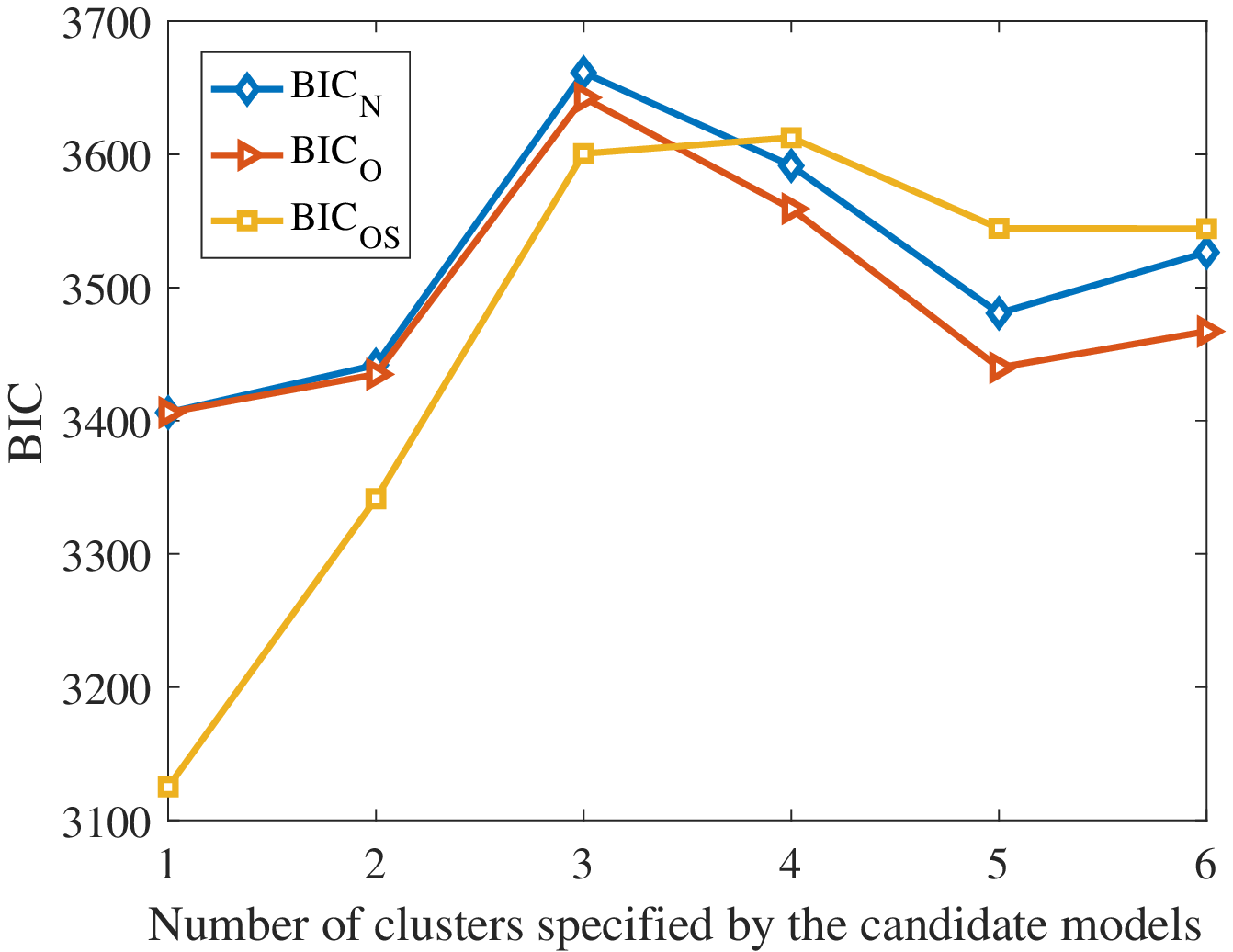}
  \caption{Total criteria}
  \label{fig:BIC}
 \end{subfigure}
 \begin{subfigure}[b]{\linewidth}
 \centering
  \includegraphics[width=\linewidth]{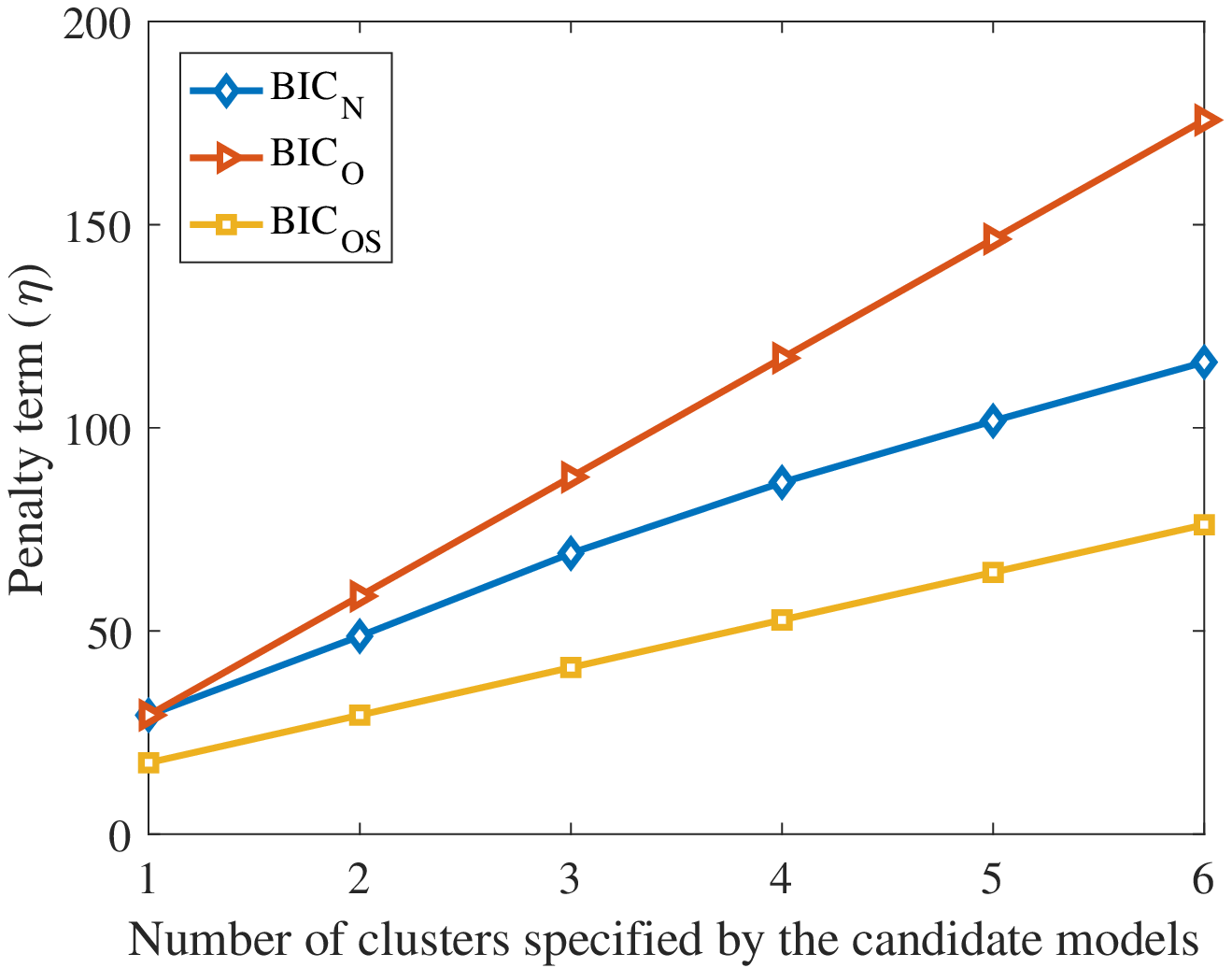} %data2Pen_v2.eps}
 \caption{Penalty terms}
  \label{fig:data2Pen}
 \end{subfigure}
 \caption{The BIC (a) and penalty term (b) of different Bayesian cluster enumeration criteria for Data-1 when $\gamma=1$.}
 \label{fig:data3_gen}
\end{figure}

\subsubsection{Initialization Strategies for Clustering Algorithms}
The overall performance of the two-step approach presented in Algorithm \ref{alg:BICG} depends on how well the clustering algorithm in the first step is able to partition the given data set. Clustering algorithms such as K-means and EM are known to converge to a local optimum and exhibit sensitivity to initialization of cluster parameters. The simplest initialization method is to randomly select cluster centroids from the set of data points. However, unless the random initializations are repeated sufficiently many times, the algorithms tend to converge to a poor local optimum. K-means++ \cite{arthur2007} attempts to solve this problem by providing a systematic initialization to K-means. One can also use a few runs of K-means++ to initialize the EM algorithm. An alternative approach to the initialization problem is to use random swap \cite{franti2018,zhao2012}. Unlike repeated random initializations, random swap creates random perturbations to the solutions of K-means and EM in an attempt to move  the clustering result away from an inferior local optimum. \\
We compare the performance of the proposed criterion and the original BIC as wrappers around the above discussed clustering methods using five synthetic data sets, which include Data-1 with $\gamma=6$, Data-2 with $N_k=500$, and the ones summarized in Table~\ref{tab:datasets}. The number of random swaps is set to $100$ and the results are an average of $100$ Monte Carlo experiments. To allow for a fair comparison, the number of replicates required by the clustering methods that use K-means++ initialization is set equal to the number of random swaps. \\
The empirical probability of detection $\left(p_{\text{det}}\right)$ of the proposed criterion and the original BIC as wrappers around the different clustering methods is depicted in Table~\ref{tab:synthresults}, where RSK-means is the random swap K-means and REM is the random swap EM. $\text{BIC}_{\mbox{\tiny NS}}$ is the implementation of the proposed BIC as a wrapper around the K-means variants and is given by 
\begin{align}
 \text{BIC}_{\mbox{\tiny NS}} &= \sum_{m=1}^lN_m\log N_m - \frac{Nr}{2}\log \hat{\sigma}^2 \nonumber \\
 &- \frac{\alpha}{2}\sum_{m=1}^l\log N_m,
 \label{eq:BICNS}
\end{align}
where $\alpha=r+1$ and $\hat{\sigma}^2$ is given by Eq.~\eqref{eq:sigma}. For the data sets that are mostly spherical, the K-means variants outperform the ones that are based on EM in terms of the correct estimation of the number of clusters, while, as expected, EM is superior for the elliptical data sets. Among the K-means variants, the gain obtained from using random swap instead of simple K-means++ is almost negligible. On the other hand, for the EM variants, EM significantly outperforms RSEM especially for $\text{BIC}_{\mbox{\tiny N}}$. 
\begin{table}[htb]
 \caption{Summary of synthetic data sets in terms of their number of features ($r$), number of samples ($N$), number of samples per cluster ($N_k$), and number of clusters ($K$).}
 \centering
 \begin{tabular}{ccccc}
  \toprule
  Data sets & $r$ & $N$ & $N_k$ &  $K$ \\
  \midrule
  S3 \cite{Ssets} & $2$& $5000$ & $333$ & $15$  \\
  A1 \cite{Asets} & $2$ & $3000$ & $150$ & $20$  \\
  G2-2-40 \cite{G2sets} & $2$ & $2048$ & $1024$ & $2$  \\
  \bottomrule
 \end{tabular}
 \label{tab:datasets}
\end{table}
\begin{table}[htb]
 \caption{Empirical probability of detection in \%.}
 \centering
 \scalebox{.97}{\begin{tabular}{cccccccc}
 \toprule
 & & Data-1 & Data-2 & S3 & A1 & G2-2-40   \\
 \midrule
 \multirow{2}{*}{K-means++ \cite{arthur2007}} & $\text{BIC}_{\mbox{\tiny NS}}$ & $49$ & $0$ & $\bm{100}$ & $98$ & $\bm{100}$ \\
 & $\text{BIC}_{\mbox{\tiny OS}}$ & $48$ & $0$ & $\bm{100}$ & $98$ & $\bm{100}$  \\
 \midrule
 \multirow{2}{*}{RSK-means \cite{franti2018}} & $\text{BIC}_{\mbox{\tiny NS}}$ & $49$ & $0$ & $\bm{100}$ & $\bm{100}$ & $\bm{100}$ \\
 & $\text{BIC}_{\mbox{\tiny OS}}$ & $48$ & $0$ & $\bm{100}$ & $\bm{100}$ & $\bm{100}$ \\
 \midrule
 \multirow{2}{*}{EM \cite{bishop2006}} & $\text{BIC}_{\mbox{\tiny N}}$ & $\bm{87}$ & $\bm{92}$ & $10$ &  $98$ & $\bm{100}$   \\
 & $\text{BIC}_{\mbox{\tiny O}}$ & $85$ & $89$ & $10$ & $98$ & $\bm{100}$  \\
 \midrule
 \multirow{2}{*}{RSEM \cite{zhao2012}} & $\text{BIC}_{\mbox{\tiny N}}$ & $22$ & $68$ & $11$ & $16$ & $90$  \\
 & $\text{BIC}_{\mbox{\tiny O}}$ & $85$ & $89$ & $9$ & $28$ & $97$ \\
  \bottomrule
 \end{tabular}}
 \label{tab:synthresults}
\end{table}

%%%%%%%%%%%%%%%%%%%%%%%%%%%%%%%
\subsection{Real Data Results}
\label{subsec:realdatares}
Although there is no randomness when repeating the experiments for the real data sets, we still use the empirical probabilities defined in Section \ref{subsec:perf} as performance measures because the cluster enumeration results vary depending on the initialization of the EM and K-means++ algorithm.
\subsubsection{Iris Data Set}
The Iris data set, also called Fisher's Iris data set \cite{fisher1936}, is a $4$-dimensional data set collected from three species of the Iris flower. It contains three clusters of 50 instances each, where each cluster corresponds to one species of the Iris flower \cite{lichman2013}. One cluster is linearly separable from the other two, while the remaining ones overlap. We have normalized the data set by dividing the features by their corresponding mean. \\ 
Fig.~\ref{fig:irisProb} shows the empirical probability of selection of different cluster enumeration criteria as a function of the number of clusters specified by the candidate models in $\mathcal{M}$. Our proposed criterion, $\text{BIC}_{\mbox{\tiny N}}$, is able to estimate the correct number of clusters $(K=3)$ $98.8\%$ of the time, while $\text{BIC}_{\mbox{\tiny O}}$ always underestimates the number of clusters to $\hat{K}_{\text{BIC}_{\mbox{\tiny O}}}=2$. $\text{BIC}_{\mbox{\tiny OS}}$ completely breaks down and, in most cases, goes for the specified maximum number of clusters. Even though two out of three clusters are not linearly separable, our proposed criterion is able to estimate the correct number of clusters with a very high empirical probability of detection. Fig.~\ref{fig:iris} shows the behavior of the BIC curves of the proposed criterion, $\text{BIC}_{\mbox{\tiny N}}$ given by Eq.~\eqref{eq:BICG}, and the original BIC implemented as a wrapper around the EM algorithm, $\text{BIC}_{\mbox{\tiny O}}$ given by Eq.~\eqref{eq:BICc}, for one Monte Carlo experiment. From Eq.~\eqref{eq:BICuni}, we know that the data-fidelity terms of both criteria are the same and this can be seen in Fig.~\ref{fig:dataIris}. But, their penalty terms are quite different, see Fig.~\ref{fig:penIris}. Due to the difference in the penalty terms of $\text{BIC}_{\mbox{\tiny N}}$ and $\text{BIC}_{\mbox{\tiny O}}$, we observe a different BIC curve in Fig.~\ref{fig:BICIris}. The total criterion (BIC) curve of $\text{BIC}_{\mbox{\tiny N}}$ has a maximum at the true number of clusters $(K=3)$, while $\text{BIC}_{\mbox{\tiny O}}$ has a maximum at $\hat{K}_{\text{BIC}_{\mbox{\tiny O}}}=2$. Observe that, again, the penalty term of our proposed criterion, $\text{BIC}_{\mbox{\tiny N}}$, has a curvature at the true number of clusters $K=3$. Just as in the simulated data experiments, the penalty term of our proposed criterion gives valuable information about the true number of clusters in the data set while the penalty terms of the other methods are uninformative on their own.
\begin{figure}[t]
 \centering
 \includegraphics[width=\linewidth]{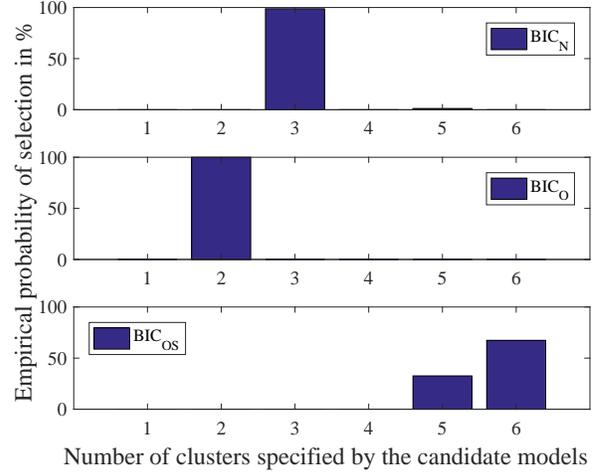} 
 \caption{Empirical probability of selection of our proposed criterion, $\text{BIC}_{\mbox{\tiny N}}$, and existing Bayesian cluster enumeration criteria for the Iris data set.} 
 \label{fig:irisProb}
\end{figure}
\begin{figure}[t]
 \centering
 \begin{subfigure}[t]{\linewidth}
 \centering
  \includegraphics[width=0.9\linewidth]{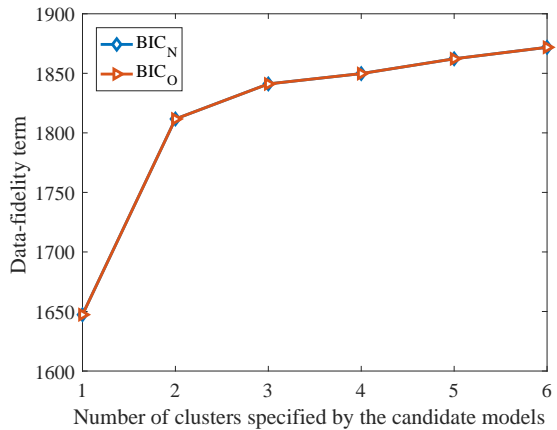}
  \caption{Data-fidelity terms}
  \label{fig:dataIris}
 \end{subfigure}
 \begin{subfigure}[t]{\linewidth}
 \centering
  \includegraphics[width=0.9\linewidth]{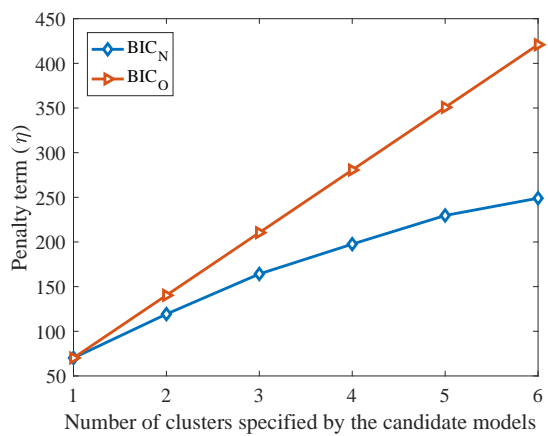}
  \caption{Penalty terms}
  \label{fig:penIris}
 \end{subfigure}
 \begin{subfigure}[t]{\linewidth}
 \centering
  \includegraphics[width=0.9\linewidth]{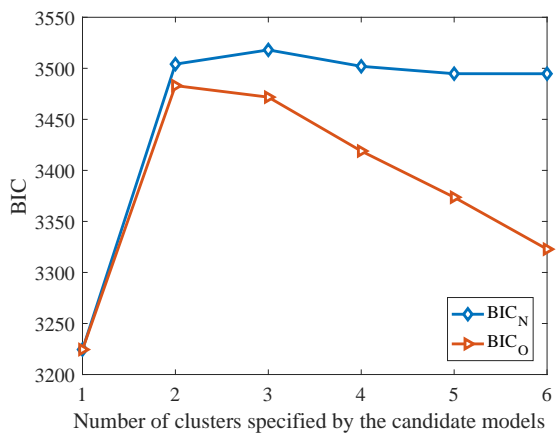}
  \caption{Total criteria}
  \label{fig:BICIris}
 \end{subfigure}
\caption{The data-fidelity term (a) penalty term (b) and the BIC (c) of the proposed criterion, $\text{BIC}_{\mbox{\tiny N}}$, and $\text{BIC}_{\mbox{\tiny O}}$ for the Iris data set.}
\label{fig:iris}
\end{figure}

\subsubsection{Multi-Object Multi-Camera Network Application}
The multi-object multi-camera network application \cite{teklehaymanot2016,binder2017} depicted in Fig.~\ref{fig:camnetwork} contains seven cameras that actively monitor a common scene of interest from different view points. There are six cars that enter and leave the scene of interest at different time frames. The video captured by each camera in the network is $18$ seconds long and $550$ frames are captured by each camera. Our objective is to estimate the total number of cars observed by the camera network. This multi-object multi-camera network example is a challenging scenario for cluster enumeration in the sense that each camera monitors the scene from different angles, which can result in differences in the extracted feature vectors (descriptors) of the same object. Furthermore, as shown in Fig.~\ref{fig:camnetwork}, the video that is captured by the cameras has a low resolution. \\
\begin{figure}[t]
 \centering
 \includegraphics[width=0.9\linewidth]{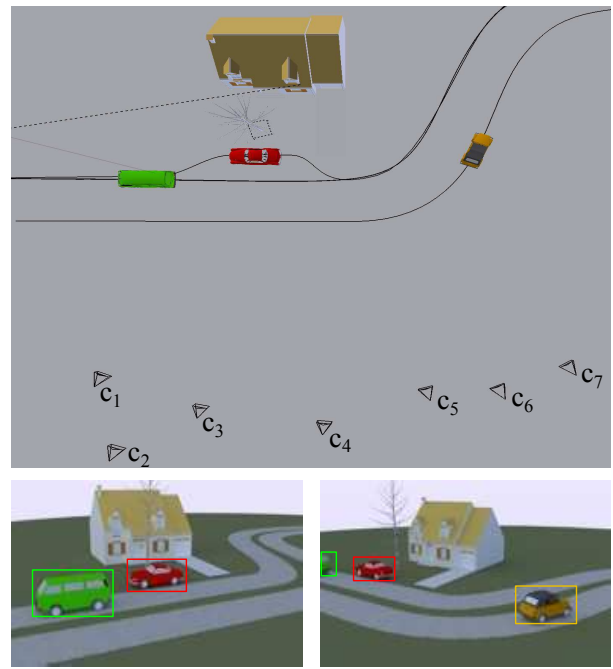}
 \caption{A wireless camera network continuously observing a common scene of interest. The top image depicts a camera network with $7$ spatially distributed cameras that actively monitor the scene from different viewpoints. The bottom left and right images show a frame captured by cameras $1$ and $7$, respectively, at the same time instant.}
 \label{fig:camnetwork}
\end{figure}
\noindent We consider a centralized network structure where the spatially distributed cameras send feature vectors to a central fusion center for further processing. Hence, each camera $c_i\in\mathcal{C}\triangleq\{c_1,\ldots,c_7\}$ first extracts the objects of interest, cars in this case, from the frames in the video using a Gaussian mixture model-based foreground detector. Then, SURF \cite{bay2008} and color features are extracted from the cars. A standard MATLAB\textsuperscript{\textcopyright} implementation of SURF is used to generate a $64$-dimensional feature vector for each detected object. Additionally, a $10$ bin histogram for each of the RGB color channels is extracted, resulting in a $30$-dimensional color feature vector. In our simulations, we apply Principal Component Analysis (PCA) to reduce the dimension of the color features to $15$. Each camera $c_i\in\mathcal{C}$ stores its feature vectors in $\mathcal{X}_{c_i}$. Finally, the feature vectors extracted by each camera, $\mathcal{X}_{c_i}$, are sent to the fusion center. At the fusion center, we have the total set of feature vectors $\mathcal{X}\triangleq\{\mathcal{X}_{c_1},\ldots,\mathcal{X}_{c_7}\}\subset\mathbb{R}^{79\times 5213}$ based on which the cluster enumeration is performed. \\   
The empirical probability of selection for different Bayesian cluster enumeration criteria as a function of the number of clusters specified by the candidate models in $\mathcal{M}$ is displayed in Fig.~\ref{fig:carProb}. Even though there are six cars in the scene of interest, two cars have similar colors. Our proposed criterion, $\text{BIC}_{\mbox{\tiny N}}$, finds six clusters only $14.7\%$ of the time, while the other cluster enumeration criteria are unable to find the correct number of clusters (cars). $\text{BIC}_{\mbox{\tiny N}}$ finds five clusters majority of the time. This is very reasonable due to the color similarity of the two cars, which results in the merging of their clusters. The original BIC, $\text{BIC}_{\mbox{\tiny O}}$, also finds five clusters majority of the time. But, it also tends to underestimate the number of clusters even more by detecting only four clusters. Hence, our proposed cluster enumeration criterion outperforms existing BIC-based methods in terms of MAE as shown in Table~\ref{tab:p+mse}, which summarizes the performance of different Bayesian cluster enumeration criteria on the real data sets. 
\begin{figure}[t]
 \centering
 \includegraphics[width=\linewidth]{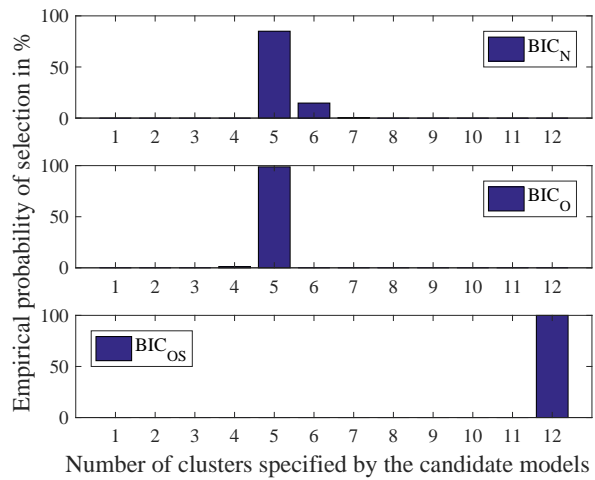} 
 \caption{Empirical probability of selection of our proposed criterion, $\text{BIC}_{\mbox{\tiny N}}$, and existing Bayesian cluster enumeration criteria for the multi-object multi-camera network application.} 
 \label{fig:carProb}
\end{figure}
\begin{table}[htb]
 \caption{Comparison of cluster enumeration performance of different Bayesian criteria for the real data sets. The performance metrics are the empirical probability of detection in $\%$, the empirical probability of underestimation in $\%$, and the Mean Absolute Error (MAE).}
 \centering
 \begin{tabular}{ccccc}
  \toprule
   & & Iris &  Cars & Seeds   \\
  \midrule
  \multirow{2}{*}{$p_{\text{det}}(\%)$} & $\text{BIC}_{\mbox{\tiny N}}$ & $\bm{98.8}$  & $\bm{14.7}$ & $\bm{100}$ \\
  & $\text{BIC}_{\mbox{\tiny O}}$ &  $0$  & $0$ & $\bm{100}$ \\
  & $\text{BIC}_{\mbox{\tiny OS}}$ &  $0$  & $0$ & $0$ \\
  \midrule
  \multirow{2}{*}{$p_{\text{under}}(\%)$} & $\text{BIC}_{\mbox{\tiny N}}$ &  $0$  & $85$ & $0$ \\
  & $\text{BIC}_{\mbox{\tiny O}}$ &  $100$    & $100$ & $0$ \\
  & $\text{BIC}_{\mbox{\tiny OS}}$ &  $0$  & $0$ & $0$ \\
  \midrule
  \multirow{2}{*}{$\text{MAE}$} & $\text{BIC}_{\mbox{\tiny N}}$   & $\bm{0.024}$   & $\bm{0.853}$ & $\bm{0}$  \\
  & $\text{BIC}_{\mbox{\tiny O}}$ &  $1$  & $1.012$ & $\bm{0}$ \\
  & $\text{BIC}_{\mbox{\tiny OS}}$ & $2.674$  & $6$  & $3$ \\
  \bottomrule
 \end{tabular}
 \label{tab:p+mse}
\end{table}

\subsubsection{Seeds Data Set} The Seeds data set is a $7$ dimensional data set which contains measurements of geometric properties of kernels belonging to three different varieties of wheat, where each variety is represented by $70$ instances \cite{lichman2013}. \\
As can be seen in Fig.~\ref{fig:seeds} the proposed criterion, $\text{BIC}_{\mbox{\tiny N}}$, and $\text{BIC}_{\mbox{\tiny O}}$ are able to estimate the correct number of clusters $100\%$ of the time while $\text{BIC}_{\mbox{\tiny OS}}$ overestimates the number of clusters to $\hat{K}_{\text{BIC}_{\mbox{\tiny OS}}}=6$. \\
In cases where either the maximum found from the BIC curve is very near to the maximum number of clusters specified by the candidate models or no clear maximum can be found, different post-processing steps that attempt to find a significant curvature in the BIC curve have been proposed in the literature. One such method is the knee point detection strategy \cite{zhao2008,zhao22008}. For the Seeds data set, applying the knee point detection method to the BIC curve generated by $\text{BIC}_{\mbox{\tiny OS}}$ allows for the correct estimation of the number of clusters $100\%$ of the time. \\
Once the number of clusters in an observed data set is estimated, the next step is to analyze the overall classification performance of the proposed two-step approach. An evaluation of the cluster enumeration and classification performance of the proposed algorithm using radar data of human gait can be found in \cite{tekleeusipco2018}. 
\begin{figure}[htb]
 \centering
 \includegraphics[width=\linewidth]{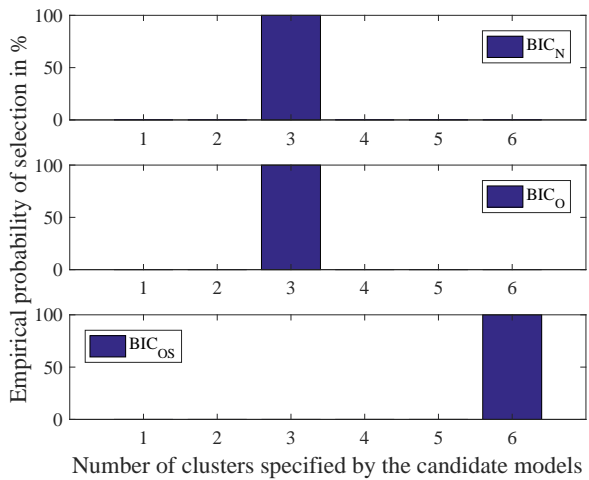}
 \caption{Empirical probability of selection of our proposed criterion, $\text{BIC}_{\mbox{\tiny N}}$, and existing Bayesian cluster enumeration criteria for the Seeds data set.}
 \label{fig:seeds}
\end{figure}

%%%%%%%%%%%%%%%%%%%%%%%%
\section{Conclusion}
\label{sec:conc}
We have derived a general expression of the BIC for cluster analysis which is applicable to a broad class of data distributions. By imposing the multivariate Gaussian assumption on the distribution of the observations, we have provided a closed-form BIC expression. Further, we have presented a two-step cluster enumeration algorithm. The proposed criterion contains information about the structure of the data in both its data-fidelity and penalty terms because it is derived by taking the cluster analysis problem into account. Simulation results indicate that the penalty term of the proposed criterion has a curvature point at the true number of clusters which is created due to the change in the trend of the curve at that point. Hence, the penalty term of the proposed criterion can contain information about the true number of clusters in an observed data set. In contrast, the penalty terms of the existing BIC-based cluster enumeration methods are uninformative on their own. In a forthcoming paper, we will alleviate the Gaussian assumption and consider robust \cite{zoubir2012,zoubir2018} cluster enumeration in the presence of outliers.
%%%%%%%%%%%%%%%%%%%%%%%%%%%%%%%%%%%%%%%%%%%%%%%%%%%%%%%
\appendices
\section{Proof of Theorem 1}
\label{app:A}
To obtain an asymptotic approximation of the FIM $\hat{\bm{J}}_m$, we first express the log-likelihood of the data points that belong to the $m$th cluster as follows:
 \begin{align}
  \log \mathcal{L}(\bm{\theta}_m|\mathcal{X}_m) &= \log \prod_{\bm{x}_n\in\mathcal{X}_m} p(\bm{x}_n\in\mathcal{X}_m) f(\bm{x}_{n}|\bm{\theta}_m) \nonumber \\
  &= \sum_{\bm{x}_n\in\mathcal{X}_m}\log\biggl(\frac{N_m}{N}\frac{1}{(2\pi)^{\frac{r}{2}}\left|\bm{\Sigma}_m\right|^\frac{1}{2}} \nonumber \\
  & \times \exp\left(-\frac{1}{2}\tilde{\bm{x}}_{n}^\top\bm{\Sigma}_m^{-1}\tilde{\bm{x}}_{n}\right)\biggr) \nonumber \\
  & = \sum_{\bm{x}_n\in\mathcal{X}_m}\biggl(\log\frac{N_m}{N}-\frac{r}{2}\log 2\pi -\frac{1}{2}\log\left|\bm{\Sigma}_m\right|\biggr) \nonumber \\
  &  -\frac{1}{2}\tr\left(\sum_{\bm{x}_n\in\mathcal{X}_m}\tilde{\bm{x}}_{n}^\top\bm{\Sigma}_m^{-1}\tilde{\bm{x}}_{n}\right) \nonumber \\
  & = N_m\log\frac{N_m}{N} -\frac{rN_m}{2}\log 2\pi -\frac{N_m}{2}\log\left|\bm{\Sigma}_m\right| \nonumber \\
  &  -\frac{1}{2}\tr\left(\bm{\Sigma}_m^{-1}\sum_{\bm{x}_n\in\mathcal{X}_m}\tilde{\bm{x}}_{n}\tilde{\bm{x}}_{n}^\top\right) \nonumber \\
  & = N_m\log\frac{N_m}{N} -\frac{rN_m}{2}\log 2\pi -\frac{N_m}{2}\log\left|\bm{\Sigma}_m\right| \nonumber \\
  & -\frac{1}{2}\tr\left(\bm{\Sigma}_m^{-1}\bm{\Delta}_m\right),
  \label{eq:loglikeliGauss}
 \end{align}
where $\tilde{\bm{x}}_{n}\triangleq \bm{x}_{n}-\bm{\mu}_m$ and $\bm{\Delta}_m \triangleq \sum_{\bm{x}_n\in\mathcal{X}_m}\tilde{\bm{x}}_{n}\tilde{\bm{x}}_{n}^\top$. Here, we have assumed that
\begin{enumerate}
 \item[\textbf{(A.6)}] the covariance matrix $\bm{\Sigma}_m, m=1,\ldots,l,$ is positive definite. 
\end{enumerate}
Then, we take the first- and second-order derivatives of $\log \mathcal{L}(\bm{\theta}_m|\mathcal{X}_m)$ with respect to the elements of $\bm{\theta}_m=\left[\bm{\mu}_m,\bm{\Sigma}_m\right]^\top$. To make the paper self contained, we have included the vector and matrix differentiation rules used in Eqs.~\eqref{eq:firstderGauss1}-\eqref{eq:secondderGauss2}, and \eqref{eq:secondderGauss3} in Appendix~\ref{app:B} (see \cite{magnus2007} for details). A generic expression of the first-order derivative of $\log \mathcal{L}(\bm{\theta}_m|\mathcal{X}_m)$ with respect to $\bm{\theta}_m$ is given by
\begin{align}
 \frac{d\log \mathcal{L}(\bm{\theta}_m|\mathcal{X}_m)}{d \bm{\theta}_m} &= -\frac{N_m}{2}\tr\left(\bm{\Sigma}_m^{-1}\frac{d\bm{\Sigma}_m}{d\bm{\theta}_m}\right) \nonumber \\
 & + \frac{1}{2}\tr\left(\bm{\Sigma}_m^{-1}\frac{d\bm{\Sigma}_m}{d\bm{\theta}_m}\bm{\Sigma}_m^{-1}\bm{\Delta}_m\right) \nonumber \\
 & + \tr\left(\bm{\Sigma}_m^{-1}\sum_{\bm{x}_n\in\mathcal{X}_m}\tilde{\bm{x}}_{n}\frac{d\bm{\mu}_m^\top}{d\bm{\theta}_m} \right) \nonumber \\
 & = \frac{1}{2}\tr\left(\frac{d\bm{\Sigma}_m}{d\bm{\theta}_m}\bm{\Sigma}_m^{-1}\bm{E}_m\bm{\Sigma}_m^{-1} \right) \nonumber \\
 & + N_m\frac{d\bm{\mu}_m^\top}{d\bm{\theta}_m}\bm{\Sigma}_m^{-1}(\bar{\bm{x}}_{m}-\bm{\mu}_m),
 \label{eq:firstderGauss1}
\end{align}
where $\bar{\bm{x}}_{m}\triangleq \frac{1}{N_m}\sum_{\bm{x}_n\in\mathcal{X}_m} \bm{x}_{n}$ is the sample mean of the data points that belong to the $m$th cluster and $\bm{E}_m\triangleq\bm{\Delta}_m-N_m\bm{\Sigma}_m$. Differentiating Eq.~\eqref{eq:firstderGauss1} with respect to $\bm{\theta}_m^\top$ results in
\begin{align}
 \frac{d^2\log \mathcal{L}(\bm{\theta}_m|\mathcal{X}_m)}{d \bm{\theta}_md\bm{\theta}_m^\top} &= \frac{1}{2}\tr\left(\frac{d\bm{\Sigma}_m}{d\bm{\theta}_m}\frac{d\bm{\Sigma}_m^{-1}}{d\bm{\theta}_m^\top}\bm{E}_m\bm{\Sigma}_m^{-1}\right) \nonumber \\
 & + \frac{1}{2}\tr\left(\frac{d\bm{\Sigma}_m}{d\bm{\theta}_m}\bm{\Sigma}_m^{-1}\bm{E}_m\frac{d\bm{\Sigma}_m^{-1}}{d\bm{\theta}_m^\top} \right) \nonumber \\
 & + \frac{1}{2}\tr\left(\frac{d\bm{\Sigma}_m}{d\bm{\theta}_m}\bm{\Sigma}_m^{-1}\frac{d\bm{E}_m}{d\bm{\theta}_m^\top}\bm{\Sigma}_m^{-1} \right) \nonumber \\
 & + N_m\frac{d\bm{\mu}_m^\top}{d\bm{\theta}_m}\frac{d\bm{\Sigma}_m^{-1}}{d\bm{\theta}_m^\top}(\bar{\bm{x}}_{m}-\bm{\mu}_m) \nonumber \\
 & - N_m\frac{d\bm{\mu}_m^\top}{d\bm{\theta}_m}\bm{\Sigma}_m^{-1}\frac{d\bm{\mu}_m}{d\bm{\theta}_m^\top} \nonumber \\
 &= \frac{N_m}{2}\tr\left(\frac{d\bm{\Sigma}_m}{d\bm{\theta}_m}\bm{\Sigma}_m^{-1}\frac{d\bm{\Sigma}_m}{d\bm{\theta}_m^\top}\bm{\Sigma}_m^{-1} \right) \nonumber \\
 & -\tr\left(\frac{d\bm{\Sigma}_m}{d\bm{\theta}_m}\bm{\Sigma}_m^{-1}\frac{d\bm{\Sigma}_m}{d\bm{\theta}_m^\top}\bm{\Sigma}_m^{-1}\bm{\Delta}_m\bm{\Sigma}_m^{-1}\right) \nonumber \\
 & - \!N_m\tr\!\left(\!\frac{d\bm{\Sigma}_m}{d\bm{\theta}_m}\bm{\Sigma}_m^{-1}(\bar{\bm{x}}_{m}\!-\!\bm{\mu}_m\!)\!\frac{d\bm{\mu}_m^\top}{d\bm{\theta}_m^\top}\bm{\Sigma}_m^{-1}\!\right) \nonumber \\
 & - N_m\frac{d\bm{\mu}_m^\top}{d\bm{\theta}_m}\bm{\Sigma}_m^{-1}\frac{d\bm{\mu}_m}{d\bm{\theta}_m^\top}.
 \label{eq:secondderGauss1}
\end{align}
Next, we exploit the symmetry of the covariance matrix $\bm{\Sigma}_m$ to come up with a final expression for the second-order derivative of $\log \mathcal{L}(\bm{\theta}_m|\mathcal{X}_m)$. The unique elements of $\bm{\Sigma}_m$ can be collected into a vector $\bm{u}_m\in\mathbb{R}^{\frac{1}{2}r(r+1)\times 1}$ as defined in \cite[pp.~56--57]{magnus2007}. Hence, incorporating the symmetry of the covariance matrix $\bm{\Sigma}_m$ and replacing the parameter vector $\bm{\theta}_m$ by $\check{\bm{\theta}}_m=\left[\bm{\mu}_m,\bm{u}_m\right]^\top$ in Eq.~\eqref{eq:secondderGauss1} results in the following expression: 
\begin{align}
 \frac{d^2\log \mathcal{L}(\bm{\theta}_m|\mathcal{X}_m)}{d \check{\bm{\theta}}_md\check{\bm{\theta}}_m^\top} &= \frac{N_m}{2}\text{vec}\left(\frac{d\bm{\Sigma}_m}{d\check{\bm{\theta}}_m}\right)^\top\bm{V}_m\text{vec}\left(\frac{d\bm{\Sigma}_m}{d\check{\bm{\theta}}_m^\top}\right) \nonumber \\
 & - \text{vec}\left(\frac{d\bm{\Sigma}_m}{d\check{\bm{\theta}}_m^\top}\right)^\top\bm{W}_m\text{vec}\left(\frac{d\bm{\Sigma}_m}{d\check{\bm{\theta}}_m}\right) \nonumber \\
 & -N_m\text{vec}\left(\frac{d\bm{\Sigma}_m}{d\check{\bm{\theta}}_m}\right)^\top\bm{Z}_m\text{vec}\left(\frac{d\bm{\mu}_m^\top}{d\check{\bm{\theta}}_m^\top}\right) \nonumber \\
 & - N_m\frac{d\bm{\mu}_m^\top}{d\check{\bm{\theta}}_m}\bm{\Sigma}_m^{-1}\frac{d\bm{\mu}_m}{d\check{\bm{\theta}}_m^\top},
\label{eq:secondderGauss2}
 \end{align}
where 
\begin{align}
 \bm{V}_m &\triangleq\bm{\Sigma}_m^{-1}\otimes\bm{\Sigma}_m^{-1} \in \mathbb{R}^{r^2\times r^2} \\
 \bm{W}_m &\triangleq\bm{\Sigma}_m^{-1}\otimes\bm{\Sigma}_m^{-1}\bm{\Delta}_m\bm{\Sigma}_m^{-1} \in \mathbb{R}^{r^2\times r^2} \\
 \bm{Z}_m &\triangleq\bm{\Sigma}_m^{-1}\left(\bar{\bm{x}}_{m}-\bm{\mu}_m\right)\otimes\bm{\Sigma}_m^{-1} \in \mathbb{R}^{r^2\times r}.
\end{align}
Eq.~\eqref{eq:secondderGauss2} can be further simplified into
\begin{align}
 \frac{d^2\log \mathcal{L}(\bm{\theta}_m|\mathcal{X}_m)}{d \check{\bm{\theta}}_md\check{\bm{\theta}}_m^\top} &= \frac{N_m}{2}\left(\frac{d\bm{u}_m}{d\bm{u}_m}\right)^\top\bm{D}^\top\bm{V}_m\bm{D}\frac{d\bm{u}_m}{d\bm{u}_m^\top} \nonumber \\
 & - \left(\frac{d\bm{u}_m}{d\bm{u}_m^\top}\right)^\top\bm{D}^\top\bm{W}_m\bm{D}\frac{d\bm{u}_m}{d\bm{u}_m} \nonumber \\
 & -N_m\left(\frac{d\bm{u}_m}{d\bm{u}_m}\right)^\top\bm{D}^\top\bm{Z}_m\text{vec}\left(\frac{d\bm{\mu}_m^\top}{d\bm{\mu}_m^\top}\right) \nonumber \\
 & - N_m\frac{d\bm{\mu}_m^\top}{d\bm{\mu}_m}\bm{\Sigma}_m^{-1}\frac{d\bm{\mu}_m}{d\bm{\mu}_m^\top},
\end{align}
where $\bm{D}\in\mathbb{R}^{r^2\times \frac{1}{2}r(r+1)}$ denotes the duplication matrix. For the symmetric matrix $\bm{\Sigma}_m$, the duplication matrix transforms $\bm{u}_m$ into $\text{vec}(\bm{\Sigma}_m)$ using the relation $\text{vec}(\bm{\Sigma}_m)=\bm{D}\bm{u}_m$ \cite[pp.~56--57]{magnus2007}. \\
A compact matrix representation of the second-order derivative of $\log\mathcal{L}(\bm{\theta}_m|\mathcal{X}_m)$ is given by
\begin{equation}
 \frac{d^2\log \mathcal{L}(\bm{\theta}_m|\mathcal{X}_m)}{d \check{\bm{\theta}}_md\check{\bm{\theta}}_m^\top} = 
 \begin{bmatrix}
  \frac{\partial^2\log \mathcal{L}(\bm{\theta}_m|\mathcal{X}_m)}{\partial \bm{\mu}_m\partial \bm{\mu}_m^\top} & \frac{\partial^2\log \mathcal{L}(\bm{\theta}_m|\mathcal{X}_m)}{\partial \bm{\mu}_m\partial \bm{u}_m^\top} \\
  \frac{\partial^2\log \mathcal{L}(\bm{\theta}_m|\mathcal{X}_m)}{\partial \bm{u}_m\partial \bm{\mu}_m^\top} & \frac{\partial^2\log \mathcal{L}(\bm{\theta}_m|\mathcal{X}_m)}{\partial \bm{u}_m\partial\bm{u}_m^\top}
 \end{bmatrix}.
\end{equation}
The individual elements of the above matrix can be written as
 \begin{align}
  \frac{\partial^2\log \mathcal{L}(\bm{\theta}_m|\mathcal{X}_m)}{\partial \bm{\mu}_m\partial \bm{\mu}_m^\top} &= -N_m\bm{\Sigma}_m^{-1}  \\
  \frac{\partial^2\log \mathcal{L}(\bm{\theta}_m|\mathcal{X}_m)}{\partial \bm{\mu}_m\partial \bm{u}_m^\top} &=-N_m\bm{Z}_m^\top\bm{D}  \\
  \frac{\partial^2\log \mathcal{L}(\bm{\theta}_m|\mathcal{X}_m)}{\partial \bm{u}_m\partial \bm{\mu}_m^\top} &= -N_m\bm{D}^\top\bm{Z}_m  \\
  \frac{\partial^2\log \mathcal{L}(\bm{\theta}_m|\mathcal{X}_m)}{\partial \bm{u}_m\partial\bm{u}_m^\top} &= \frac{N_m}{2}\bm{D}^\top\bm{F}_m\bm{D},
 \label{eq:secondderGauss3}
 \end{align}
where $\bm{F}_m\triangleq\bm{\Sigma}_m^{-1}\otimes\left(\bm{\Sigma}_m^{-1}-\frac{2}{N_m}\bm{\Sigma}_m^{-1}\bm{\Delta}_m\bm{\Sigma}_m^{-1}\right) \in\mathbb{R}^{r^2\times r^2}$. \\
The FIM of the $m$th cluster can be written as
\begin{align}
  \hat{\bm{J}}_m &= 
 \begin{bmatrix}
  -\frac{\partial^2\log \mathcal{L}(\hat{\bm{\theta}}_m|\mathcal{X}_m)}{\partial \bm{\mu}_m\partial \bm{\mu}_m^\top} & -\frac{\partial^2\log \mathcal{L}(\hat{\bm{\theta}}_m|\mathcal{X}_m)}{\partial \bm{\mu}_m\partial \bm{u}_m^\top} \\
  -\frac{\partial^2\log \mathcal{L}(\hat{\bm{\theta}}_m|\mathcal{X}_m)}{\partial \bm{u}_m\partial \bm{\mu}_m^\top} & -\frac{\partial^2\log \mathcal{L}(\hat{\bm{\theta}}_m|\mathcal{X}_m)}{\partial \bm{u}_m\partial\bm{u}_m^\top}
 \end{bmatrix} \nonumber \\
 &= \begin{bmatrix}
     N_m\hat{\bm{\Sigma}}_m^{-1} & N_m\hat{\bm{Z}}_m^\top\bm{D} \\
     N_m\bm{D}^\top\hat{\bm{Z}}_m & -\frac{N_m}{2}\bm{D}^\top\hat{\bm{F}}_m\bm{D}
    \end{bmatrix}.
\label{eq:j2}
\end{align}
The maximum likelihood estimators of the mean and covariance matrix of the $m$th Gaussian cluster are given by
\begin{align}
 \hat{\bm{\mu}}_m &= \frac{1}{N_m}\sum_{\bm{x}\in\mathcal{X}_m}\bm{x}_{n} = \bar{\bm{x}}_{m} \\
  \hat{\bm{\Sigma}}_m &= \frac{1}{N_m}\sum_{\bm{x}_n\in\mathcal{X}_m}\left(\bm{x}_{n}-\hat{\bm{\mu}}_m\right)\left(\bm{x}_{n}-\hat{\bm{\mu}}_m\right)^\top.
\end{align}
Hence, $\hat{\bm{Z}}_m \triangleq\hat{\bm{\Sigma}}_m^{-1}\left(\bar{\bm{x}}_{m}-\hat{\bm{\mu}}_m\right)\otimes\hat{\bm{\Sigma}}_m^{-1} = \bm{0}_{r^2\times r}$. Consequently, Eq.~\eqref{eq:j2} can be further simplified to
\begin{equation}
\hat{\bm{J}}_m = 
 \begin{bmatrix}
     N_m\hat{\bm{\Sigma}}_m^{-1} & \bm{0}_{r\times \frac{1}{2}r(r+1)} \\
     \bm{0}_{\frac{1}{2}r(r+1)\times r} & -\frac{N_m}{2}\bm{D}^\top\hat{\bm{F}}_m\bm{D}
 \end{bmatrix}.
\end{equation}
The determinant of the FIM, $\hat{\bm{J}}_m$, can be written as
\begin{equation}
 \left|\hat{\bm{J}}_m\right| = \left|N_m\hat{\bm{\Sigma}}_m^{-1}\right|\times \left|-\frac{N_m}{2}\bm{D}^\top\hat{\bm{F}}_m\bm{D}\right|. 
\end{equation}
As $N\rightarrow\infty$, $N_m\rightarrow\infty$ given that $l<<N$, it follows that
\begin{equation}
 \left|\frac{1}{N_m}\hat{\bm{J}}_m\right| \approx \mathcal{O}(1),
\label{eq:fisherinfo2}
\end{equation}
where $\mathcal{O}(1)$ denotes Landau's term which tends to a constant as $N\rightarrow\infty$. Using this result, we provide an asymptotic approximation to Eq.~\eqref{eq:BICgen}, in the case where $\mathcal{X}$ is composed of Gaussian distributed data vectors, as follows:
\begin{align}
 \log p(M_l|\mathcal{X}) &\approx \log p(M_l) + \log f(\hat{\bm{\Theta}}_l|M_l) + \log \mathcal{L}(\hat{\bm{\Theta}}_l|\mathcal{X}) \nonumber \\
 & + \frac{lq}{2}\log 2\pi - \frac{1}{2}\sum_{m=1}^l\log \left|N_m\frac{1}{N_m}\hat{\bm{J}}_m\right| + \rho \nonumber \\
 & = \log p(M_l) + \log f(\hat{\bm{\Theta}}_l|M_l) + \log \mathcal{L}(\hat{\bm{\Theta}}_l|\mathcal{X}) \nonumber \\
 & + \frac{lq}{2}\log 2\pi - \frac{q}{2}\sum_{m=1}^l\log N_m \nonumber \\
 & - \frac{1}{2}\sum_{m=1}^l\log \left|\frac{1}{N_m}\hat{\bm{J}}_m\right| + \rho,
\label{eq:gaussianBIC}
\end{align}
where $q=\frac{1}{2}r(r+3)$. Assume that
\begin{enumerate}
 \item[\textbf{(A.7)}] $p(M_l)$ and $f(\hat{\bm{\Theta}}_l|M_l)$ are independent of the data length $N$.
\end{enumerate}
Then, ignoring the terms in Eq.~\eqref{eq:gaussianBIC} that do not grow as $N\rightarrow\infty$ results in
\begin{align}
 \text{BIC}_{\mbox{\tiny N}}(M_l) &\triangleq \log p(M_l|\mathcal{X}) \nonumber \\
 & \approx \log \mathcal{L}(\hat{\bm{\Theta}}_l|\mathcal{X}) - \frac{q}{2}\sum_{m=1}^l\log N_m + \rho.
\end{align}
Since $\mathcal{X}$ is composed of multivariate Gaussian distributed data, $\text{BIC}_{\mbox{\tiny N}}(M_l)$ can be further simplified as follows:
\begin{align}
 \text{BIC}_{\mbox{\tiny N}}(M_l) &\!=\! \log \mathcal{L}(\hat{\bm{\Theta}}_l|\mathcal{X}) + p_l \nonumber \\
 &= \sum_{m=1}^l\biggl(N_m\log\frac{N_m}{N}-\frac{rN_m}{2}\log2\pi\nonumber \\
 & -\frac{N_m}{2}\log\left|\hat{\bm{\Sigma}}_m\right| -\frac{1}{2}\tr\left(N_m\hat{\bm{\Sigma}}_m^{-1}\hat{\bm{\Sigma}}_m\right)\biggr) + p_l \nonumber \\
 &= \sum_{m=1}^l N_m\log N_m - N\log N - \frac{rN}{2}\log 2\pi \nonumber \\ 
 & - \sum_{m=1}^l\frac{N_m}{2}\log \left|\hat{\bm{\Sigma}}_m\right| - \frac{rN}{2} + p_l,
\label{eq:normalBIC}
 \end{align}
where 
\begin{equation}
 p_l\triangleq - \frac{q}{2}\sum_{m=1}^l\log N_m + \rho.
\end{equation}
Finally, ignoring the model independent terms in Eq.~\eqref{eq:normalBIC} results in Eq.~\eqref{eq:BICG} which concludes the proof. 

\section{Vector and Matrix Differentiation Rules}
\label{app:B}
Here, we describe the vector and matrix differentiation rules used in this paper (see \cite{magnus2007} for details). Let $\bm{\mu}\in\mathbb{R}^{r\times 1}$ be the mean and $\bm{\Sigma}\in\mathbb{R}^{r\times r}$ be the covariance matrix of a multivariate Gaussian random variable $\bm{x}$. Assuming that the covariance matrix $\bm{\Sigma}$ has no special structure, the following vector and matrix differentiation rules hold.
\begin{align}
\frac{d}{d\bm{\Sigma}}\log|\bm{\Sigma}| &= \tr\left(\bm{\Sigma}^{-1}\frac{d\bm{\Sigma}}{d\bm{\Sigma}}\right) \\
\frac{d}{d\bm{\Sigma}}\tr\left(\bm{\Sigma}\right) &= \tr\left(\frac{d\bm{\Sigma}}{d\bm{\Sigma}}\right) \\
\frac{d}{d\bm{\Sigma}}\bm{\Sigma}^{-1} &= -\bm{\Sigma}^{-1}\frac{d\bm{\Sigma}}{d\bm{\Sigma}}\bm{\Sigma}^{-1} \\
\frac{d}{d\bm{\mu}}\bm{\mu}^\top\bm{\mu} &= 2\bm{\mu}^\top
\end{align}
Given three arbitrary symmetric matrices $\bm{A}$, $\bm{B}$, and $\bm{Y}$ with matching dimensions
\begin{align}
 \tr\left(\bm{A}\bm{Y}\bm{B}\bm{Y}\right) &= \text{vec}(\bm{Y})^\top(\bm{A}\otimes\bm{B})\text{vec}(\bm{Y}) \\
 &= \bm{u}^\top\bm{D}^\top(\bm{A}\otimes\bm{B})\bm{D}\bm{u},
\end{align}
where $\bm{u}$ contains the unique elements of the symmetric matrix $\bm{Y}$ and $\bm{D}$ denotes the duplication matrix of $\bm{Y}$. In Eq.~\eqref{eq:secondderGauss3} we used the relation $\text{vec}\left(\frac{d\bm{Y}}{d\bm{u}}\right) = \bm{D}\frac{d\bm{u}}{d\bm{u}}$.

\section*{Acknowledgment}
We thank the anonymous reviewers for their insightful comments and suggestions. Further, we would like to thank Dr. Benjamin B\'{e}jar Haro for providing us with the multi-object multi-camera data set which was created as a benchmark with in the project HANDiCAMS. HANDiCAMS acknowledges the financial support of the Future and Emerging Technologies (FET) programme within the Seventh Framework Programme for Research of the European Commission, under FET-Open grant number: 323944.  The work of F. K. Teklehaymanot is supported by the `Excellence Initiative' of the German Federal and State Governments and the Graduate School of Computational Engineering at Technische Universit\"at Darmstadt and by the LOEWE initiative (Hessen, Germany) within the NICER project. The work of M. Muma is supported by the `Athene Young Investigator Programme' of Technische Universit\"at Darmstadt.

\bibliographystyle{IEEEtran}
\bibliography{IEEEabrv,myReferences}

\end{document}